\documentclass[12pt,a4paper,english]{smfart}
\usepackage[english]{babel}
\usepackage{smfthm, mathabx}
\usepackage{amssymb}
\usepackage{amsmath}
\usepackage{amsthm}
\usepackage{amsfonts}
\usepackage{graphicx}
\usepackage{enumerate}

\usepackage{tikz-cd}
\usepackage{appendix}

\usepackage{euscript,mathrsfs}
\usepackage[utf8]{inputenc} 
\usepackage{longtable}
\usepackage{dsfont}

\usepackage[OT2,T1]{fontenc}

\usepackage[all,cmtip]{xy}
\usepackage{alltt}
\usepackage{comment}

\usepackage{calrsfs}

\xyoption{all}

\usepackage{float}

\usepackage{fullpage}

\usepackage{color}

\newtheorem{Theorem}{Theorem}
\newtheorem{Exa}{Example}
\newtheorem*{thm}{Theorem}

\author[Oussama Hamza]{Oussama Hamza}
\address{Department of Mathematics\\ Western University\\ London\\ Ontario\\ Canada N6A5B7}
\email{ohamza3@uwo.ca}

\title{Zassenhaus and lower central filtrations of Pro-$p$ groups considered as modules}

\thanks{I first praise Christian Maire for several careful readings, corrections, improvement and inspiration. I also acknowledge Simion Filip, John Labute, J{\'a}n Min{\'a}{\v c}, Nguyên Duy Tân and Thomas Weigel for discussions, suggestions and references. I thank Elyes Boughattas, Baptiste Cerclé and Michael Rogelstad for useful comments. Finally, I am grateful to the anonymous referee for provided comments.}

\newcommand{\Q}{\mathbb{Q}}
\newcommand{\F}{\mathbb{F}}
\newcommand{\Z}{\mathbb{Z}}
\newcommand{\NN}{\mathbb{N}}

\newcommand{\C}{\mathbb{C}}

\def\p{{\mathfrak p}}

\def\rk{{\rm rank}}

\def\log{{\rm log}}

\def\grad{{\rm Grad}}

\def\Gal{{\rm Gal}}
\def\cd{{\rm cd}}

\def\AA{{\mathbb A}}

\def\O{{\mathcal O}}

\def\E{{\mathcal E}}

\def\I{{\mathcal I}}
\def\Ff{{\mathcal F}}

\def\Ll{{\mathcal L}}

\def\L{{\rm L}}
\def\K{{\rm K}}
\def\k{{\rm k}}

\def\p{{\mathfrak p}}

\begin{document}

\begin{abstract}
The goal of this paper is to study the action of groups on Zassenhaus and lower central filtrations of finitely generated pro-$p$ groups. We shall focus on the semisimple case. Particular attention is given to finitely presented groups of cohomological dimension less than or equal to two.
\end{abstract}

\maketitle

\section*{Introduction}

\subsection*{Context}
Let~$G$ be a finitely generated pro-$p$ group, and denote by~$\AA$ the ring~$\Z_p$ or~$\F_p$. From~$\AA$, we recover some filtrations on~$G$. Introduce~$Al(\AA, G):={\varprojlim} {\AA}\lbrack G/U \rbrack$, where~$U$ is an open normal subgroup of~$G$, the completed group algebra of~$G$ over~$\AA$. Since~$\AA\lbrack G/U\rbrack$ is an augmented algebra over~$\AA$, then~$Al(\AA,G)$ is also. Consequently, we denote by~$Al_n(\AA, G)$ the~$n$-th power of the augmentation ideal of~$Al(\AA, G)$. Define:
$$G_n(\AA):=\{g\in G; g-1\in Al_n(\AA,G)\},$$
this is a filtration of~$G$.

Observe that~$\{G_n(\F_p)\}_{n\in \NN}$ denotes the Zassenhaus filtration of~$G$ (see for instance \cite{Minac}), and is an open characteristic basis of~$G$. Similarly, under certain conditions (see \cite{Labutecentralseries}), the filtration~$\{G_n(\Z_p)\}_{n\in \NN}$ is equal to the lower central series of~$G$, i.e.~$G_1(\Z_p)=G$ and~$G_{n+1}(\Z_p)=\lbrack G_n(\Z_p);G\rbrack$. When the context is clear, we omit to write~$\AA$ for filtrations (and future associated invariants). Our goal is to study the following Lie algebras:
\begin{equation*}
\begin{aligned}
\Ll(\AA,G):=\bigoplus_{n\in \NN} \Ll_n(\AA,G), \quad \text{where} \quad \Ll_n(\AA,G):=G_n(\AA)/G_{n+1}(\AA), \quad \text{and} 
\\ \E(\AA,G):=\bigoplus_{n\in \NN} \E_n(\AA,G), \quad \text{where} \quad \E_n(\AA, G):=Al_n(\AA,G)/Al_{n+1}(\AA,G).
\end{aligned}
\end{equation*}

We always assume that~$\Ll(\AA,G)$ is \textbf{torsion-free} over~$\AA$. Notice that this condition is automatically satisfied when~$\AA:=\F_p$, contrary to the case~$\AA:=\Z_p$ (see for instance \cite[Theorem]{LABUTE197016} and \cite[Th{\'e}or{\`e}me $2$]{labute1967algebres}). Since~$G$ is finitely generated, one defines for every integer~$n$:
\begin{equation*}
\begin{aligned}
a_n(\AA):=\rk_{\AA}\Ll_n(\AA,G), \quad \text{and} \quad c_n(\AA):=\rk_{\AA} \E_n(\AA,G),
\\ gocha(\AA, t):=\sum_{n\in \NN} c_nt^n.
\end{aligned} 
\end{equation*}

The series~$gocha(\F_p,t)$ was first introduced by Golod and Shafarevich in \cite{Gold}. It allowed them to obtain information on class field towers over some fields (see for instance \cite[Chapter IX]{CA}). Later in~$1965$, Lazard studied analytic pro-$p$ groups in \cite{LAZ}, i.e. Lie groups over~${\Q_p}$ (see \cite[Définition~$3.1.2$]{LAZ}). Labute \cite{Labute}, also used the series~$gocha(\Z_p,t)$ in order to study mild groups and their related properties.

Jennings, Lazard and Labute gave an explicit relation between~$gocha$ and~$(a_n)_{n\in \NN}$ (\cite[Proposition~$3.10$, Appendice A]{LAZ}, and  \cite[Lemma~$2.10$]{Minac}):
\begin{equation}\label{JeLa}
\begin{aligned}
& gocha(\AA,t)=\prod_{n\in \NN} P(\AA,t^n)^{a_n(\AA)},&
\\ \text{where} \quad & P(\F_p,t):=\left(\frac{1-t^{p}}{1-t}\right), \quad \text{and} & \quad   P(\Z_p,t):=\left(\frac{1}{1-t}\right).
\end{aligned}
\end{equation}

From Formula \eqref{JeLa}, Lazard deduced an alternative for the growth of~$(c_n(\F_p))_{n\in \NN}$ (for general references, see \cite[Part~$12.3$]{DDMS}), this is \cite[Théorème~$3.11$, Appendice A.$3$]{LAZ}:

\begin{thm}[Alternative des Gocha]
We have the following alternative:
\begin{itemize}
\item[$\bullet$]
Either~$G$ is an analytic pro-$p$ group, so there exists an integer~$n$ such that~$a_n(\F_p)=0$ and the sequence~$(c_n(\F_p))_{n\in \NN}$ has polynomial growth with~$n$.
\item[$\bullet$] Or~$G$ is not an analytic pro-$p$ group, then for every~$n\in \NN$,~$a_n(\F_p)\neq 0$, and the sequence~$(c_n(\F_p))_{n\in \NN}$ does admit an exponential growth with~$n$ (i.e. grows faster than any polynomial in~$n$). 
\end{itemize}
\end{thm}

In~$2016$, Min{\'{a}}{\v{c}}, Rogelstad and Tân \cite{Minac} improved Formula \eqref{JeLa}: they gave an explicit relation between the sequences~$(a_n)_{n\in \NN}$ and~$(c_n)_{n\in \NN}$. The main idea was to introduce the coefficients~$b_n\in \Q$, namely defined by:
\begin{equation*}
\log(gocha(\AA,t)):=\sum_{n\in \NN}b_n(\AA) t^n.
\end{equation*}
They obtained the following formula (\cite[Theorem $2.9$ and Lemma $2.10$]{Minac}): if we write~$n=mp^k$, with~$m$ coprime to~$p$, then
\begin{equation}\label{minfor}
\begin{aligned}
a_n(\F_p)=w_m(\F_p)+w_{mp}(\F_p)+\dots+w_{mp^k}(\F_p),  \quad a_n(\Z_p)=w_n(\Z_p);
\\ \text{ where } w_n(\AA):=\frac{1}{n}\sum_{m|n} \mu(n/m)mb_m(\AA) \quad \text{and} \quad \mu \text{ is the Möbius function}.
\end{aligned}
\end{equation}

Notice that the number~$w_n(\F_p)$ (resp.~$c_n(\Z_p), a_n(\Z_p)$) is denoted by~$w_n(G)$ (resp.~$d_n(G)$,~$e_n(G)$) in~\cite[Part $2$]{Minac}. Furthermore, Min{\'a}{\v c}, Rogelstad and Tân asked the following question,~\cite[Question~$2.13$]{Minac}:
\begin{quote}
\emph{Do we have~$c_n(\F_p):=c_n(\Z_p)$?}
\end{quote} 
\medskip

Theorem \ref{MRTsol} answers this question positively when~$G$ is finitely presented and~$\cd(G)\leq 2$. To proceed, we compute~$(c_n(\AA))_{n\in \NN}$ by the Lyndon resolution (see \cite[Corollay~$5.3$]{BRU}), and as a consequence, we infer an explicit formula for~$a_n(\AA)$ using Formula \eqref{minfor}. Weigel (\cite[Theorem D]{weigel2015graded}) also gave a different formula from \eqref{minfor}, involving~$a_n(\Z_p)$ and roots of~$1/gocha(\Z_p,t)$.

\subsection*{Statement of main results}
The goal of this paper is to extend equations \eqref{JeLa},~\eqref{minfor} and Gocha's alternative in an equivariant context. We use here the terminology equivariant to stress the action of groups.

Let~$q$ be a prime dividing~$p-1$, and assume that~$Aut(G)$ contains a cyclic subgroup~$\Delta$ of order~$q$. We denote by~${\rm Irr}(\Delta)$ the group of~$\AA$-irreducible characters~$\chi$ of~$\Delta$, with trivial character~$\mathds{1}$: this is a group of order $q$ which does not depend on the choice of $\AA$ (for general references on~$\AA$-characters, see \cite[Chapitre~$14$]{Ser}). If~$M$ is a~$\AA\lbrack \Delta \rbrack$-module, one defines the eigenspaces of $M$ by:
$$M\lbrack\chi\rbrack:=\{x\in M;\quad \forall \delta\in \Delta, \quad \delta(x)={\chi}(\delta) x\}.$$

Notice that~$\Ll_n(\AA,G)$ and~$\E_n(\AA,G)$ are free, finitely generated over ~$\AA$ and are~$\AA\lbrack \Delta\rbrack$-modules. We study the following quantities:
\begin{equation*}
\begin{aligned}
a_n^{\chi}(\AA):=\rk_\AA\Ll_n(\AA,G)\lbrack \chi\rbrack,\quad \text{and} \quad c_n^{\chi}(\AA):=\rk_\AA\E_n(\AA,G)\lbrack \chi \rbrack.
\end{aligned}
\end{equation*}
From Maschke's Theorem and \cite[Partie $14.4$]{Ser}, we obtain the following equality:
$$a_n(\AA)=\sum_{\chi \in {\rm Irr}(\Delta)}a_n^{\chi}(\AA), \quad \text{and} \quad c_n(\AA)=\sum_{\chi \in {\rm Irr}(\Delta)}c_n^{\chi}(\AA).$$
\medskip

This article has three parts.
\medskip 

Part \ref{rep} is mostly inspired by arguments given in \cite{Minac}. Denote by $R\lbrack \Delta \rbrack$ the finite representation ring of~$\Delta$ over~$\AA$. Observe that~$R[\Delta]$ is a ring isomorphic to~$\Z[{\rm Irr}(\Delta)]$, consequently~$R[\Delta]\bigotimes_{\Z}\Q$ is a~$\Q$-algebra isomorphic to~$\Q[{\rm Irr}(\Delta)]$. Instead of considering series with coefficients in~$\Q$, as Filip \cite{Fi} and Stix \cite{stix2013rational} did, we study series with coefficients in~$R\lbrack \Delta \rbrack \bigotimes_\Z \Q$. Let us introduce:
\begin{equation*}
gocha^\ast(\AA,t):=\sum_{n\in \NN} \left(\sum_{\chi \in {\rm Irr}(\Delta)} c_n^{\chi}(\AA){\chi}\right) t^n.
\end{equation*}
We infer an equivariant version of the equality \eqref{JeLa}:

\begin{Theorem}\label{Jen}
The following equality holds for series with coefficients in~$R\lbrack \Delta \rbrack$:
\begin{equation*}
\begin{aligned}
gocha^\ast(\AA,t)=\prod_{n\in \NN} \prod_{\chi \in {\rm Irr}(\Delta)}P_\chi(\AA, t^n)^{a_n^{\chi}(\AA)},
\\ \text{ where } \quad P_\chi(\F_p, t):=\frac{1- \chi.t^p}{1-\chi.t}, \quad \text{and} \quad P_\chi(\Z_p, t):=\frac{1}{1-\chi.t}
\end{aligned}
\end{equation*}
\end{Theorem}

As done in Part~$2$ \cite{Minac}, one introduces the logarithm of series with coefficients in~$R\lbrack \Delta \rbrack$, defined by rationals~$b_n^{\chi}(\AA)\in \Q$:
\begin{equation*}
log(gocha^\ast(\AA,t)):=\sum_{n\in \NN} \left(\sum_{\chi \in {\rm Irr}(\Delta)} b_n^{\chi}(\AA){\chi}\right) t^n.
\end{equation*}
Then, we obtain an equivariant version of Formula \eqref{minfor}.

Write~$n:=mp^k$, with~$(m,p)=1$, and assume~$(n,q)=1$. Then:
\begin{equation}\label{ap}
\begin{aligned}
a_n^{\chi}(\F_p)=w_m^{\chi}(\F_p)+w_{mp}^{\chi}(\F_p)+\dots+w_{mp^k}^{\chi}(\F_p), \quad \text{ and } \quad a_n^{\chi}(\Z_p)=w_n^{\chi}(\Z_p),
\\ \text{ where } \quad w_n^{\chi}(\AA):=\frac{1}{n}\sum_{m|n}\mu(n/m)mb_m^{\chi^{m/n}}(\AA)\in \Q.
\end{aligned}
\end{equation}

Some results on the coefficients~$a_n^{\chi}(\Z_p)$ were given by Filip \cite{Fi} and Stix \cite{stix2013rational} for groups defined by one quadratic relation.
\medskip

In Part \ref{assy}, we study cardinalities of eigenspaces of~$\Ll(\AA, G)$. When~$\Ll(\AA, G)$ is infinite dimensional (as a free module over~$\AA$), we observe using the pigeonhole principle that~$\Ll(\AA, G)$ admits at least one infinite dimensional eigenspace.
\begin{quote}
\emph{Main Question: Which eigenspaces of~$\Ll(\AA, G)$ are infinite dimensional?}
\end{quote}
For this purpose, we introduce~$\chi_0$-filtration on~$Al(\AA, G)$, which depends on a fixed non-trivial irreducible character~$\chi_0$. It is denoted by~$(E_{\chi_0, n}(\AA,G))_n$, and described in Subpart~\ref{DfilLamb}. Furthermore, we assume that $E_{\chi_0, n}(\AA, G)/E_{\chi_0,n+1}(\AA, G)$ is \textbf{torsion-free} over $\AA$. This condition is automatically satisfied when $\AA=\F_p$; and for $\AA=\Z_p$, it is true whenever $G$ is free or in the situation of \cite[Th{\'e}or{\`e}me $2$]{labute1967algebres}.
This allows us to define~$gocha_{\chi_0}(\AA,t)$ by:
\begin{equation*}
\begin{aligned}
gocha_{\chi_0}(\AA,t):=\sum_{n\in \NN} c_{\chi_0, n}(\AA) t^n, 
\\ \text{where} \quad c_{\chi_0, n}(\AA):=\rk_\AA (E_{\chi_0, n}(\AA, G)/E_{\chi_0,n+1}(\AA, G)).
\end{aligned}
\end{equation*}

\medskip

Part \ref{expart} illustrates our theoretical results for finitely presented pro-$p$ groups~$G$, with cohomological dimension~$\cd(G)$ less than or equal to~$2$. 

Proposition \ref{key} allows us to compute the~$gocha$ series of~$G$, and shows that the inverse of the gocha series is a polynomial. Then Theorem \ref{MRTsol} answers (and generalizes) \cite[Question~$2.13$]{Minac}, showing that~$gocha(\AA,t)$,~$gocha^\ast(\AA,t)$ and~$gocha_{\chi_0}(\AA,t)$ do not depend on the choice of the ring~$\AA$. Finally, considering \cite[Theorem D]{weigel2015graded} in our context, one recovers~$a_n^\chi$ from roots of the polynomial~$1/gocha^\ast$ (see Proposition \ref{Weigen}).

Let us now introduce our last result. Since (Proposition \ref{key})~$\chi_{eul,\chi_0}(t):=1/gocha_{\chi_0}(t)$ is a polynomial, we write the degree of~$\chi_{eul,\chi_0}$ as~$\deg_{\chi_0}(G)$. Denote the~$\chi_0$-eigenvalues of~$G$ by~$\lambda_{\chi_0,i}$, and let~$L_{\chi_0}(G)$ be the~$\chi_0$-entropy of~$G$ defined by:
$$\chi_{eul,\chi_0}(t):=\prod_{i=1}^{\deg_{\chi_0}(G)}(1-\lambda_{\chi_0,i}t),\quad L_{\chi_0}(G):= \max_{1\leq i \leq \deg_{\chi_0}(G)} |\lambda_{\chi_0,i}|.$$

\begin{Theorem} \label{mildcomp2D}
Assume for some~$\chi_0$ that~$L_{\chi_0}(G)$ is reached for a unique eigenvalue~$\lambda_{\chi_0,i}$ such that:
\begin{enumerate}[\quad (i)]
\item~$\lambda_{\chi_0,i}$ is real,
\item~$L_{\chi_0}(G)=\lambda_{\chi_0,i}>1$.
\end{enumerate}
Then every eigenspace of~$\Ll(\AA, G)$ is infinite dimensional.
\end{Theorem}

We also prove in Theorem \ref{freecomp2}, that every finitely generated noncommutative free pro-$p$ group~$G$ satisfies the hypotheses of Theorem \ref{mildcomp2D}. Let us finish this introduction with explicit examples:


\begin{Exa}[Cohomological dimension~$2$]\label{mildex}
Let us take~$p=103$,~$q=17$. Observe that~$\overline{8}\in {\F_{103}}$ is a primitive~$17$-th root of unity.
\\Consider the pro-$103$ group~$G$, generated by three generators~$x,y,z$ and one relation~$u=\lbrack x;y \rbrack$. By \cite[Theorem]{LABUTE197016}, the~$\Z_p$-module~$\Ll(\Z_p,G)$ is torsion-free. If we apply \cite[Corollary~$5.3$]{FOR} and Proposition \ref{key}, we remark that~$\cd(G)=2$ and: 
$$gocha(\AA,t):=1/(1-3t+t^2).$$ 
Introduce an automorphism~$\delta$ on~$G$, by~$\delta(x):=x^8$,~$\delta(y):=y^{8^2}$ and~$\delta(z):=z^{8^3}$; Proposition \ref{com} justifies that this action is well defined. Consequently~$\Delta:=\langle \delta \rangle$ is a subgroup of order~$17$ of~$Aut(G)$. Fix the character~${\chi_0}\colon \Delta \to \F_{103}^\times; \delta \mapsto \overline{8}$.

Applying Formula \eqref{ap}, let us compute some coefficients~$a_n^{\chi}$ and~$c_n^{\chi}$.
\\Observe first that: 
\begin{multline*}
gocha^\ast(\AA,t):=\frac{1}{1-({\chi_0}+{\chi_0}^2+{\chi_0}^3).t+{\chi_0}^3. t^2},\quad \text{and} \quad 
\\ \log(gocha^\ast(\AA,t))=({\chi_0}+{\chi_0}^2+{\chi_0}^3).t+
({\chi_0}^6/2+{\chi_0}^5+\frac{3\chi_0^4}{2}+\frac{\chi_0^2}{2}).t^2+
\\(\frac{\chi_0^9}{3}+{\chi_0}^8+2{\chi_0}^7+\frac{4\chi_0^6}{3}+\chi_0^5+\frac{\chi_0^3}{3}).t^3+\dots .
\end{multline*}
From Formula \eqref{minfor}, we infer:~$a_2=2$ and~$a_3=5$. Furthermore Formula \eqref{ap} gives us: 
\begin{equation*}
a_2^{\chi_0^i}=\frac{2b_2^{\chi_0^i}-b_1^{\chi_0^{9i}}}{2}, \quad \text{and} \quad a_3^{\chi_0^i}=\frac{3b_3^{\chi_0^i}-b_1^{\chi_0^{6i}}}{3}.
\end{equation*}
Consequently, we obtain:
\begin{itemize}
\item[$\bullet$]~$a_2^{\chi_0^4}=a_2^{\chi_0^5}=1$, else~$a_2^{\chi_0^i}=0$ when~$i\neq 5$.
\item[$\bullet$]~$a_3^{\chi_0^8}=a_3^{\chi_0^6}=a_3^{\chi_0^5}=1$, and~$a_3^{\chi_0^7}=2$. Else if~$i\notin \{5;6;7;8\}$,~$a_3^{\chi_0^i}=0$.
\end{itemize}

Here, by \cite[Theorem $1$ and Part $3$]{Labutecentralseries}, the algebra $\Ll_{\chi_0}(G,\Z_p)$ is torsion-free over $\Z_p$. We have:~$$gocha_{\chi_0}(\AA,t):=\frac{1}{1-t-t^2},$$
and the maximal~$\chi_0$-eigenvalue of~$G$ is real and strictly greater than~$1$. 
\\Therefore by Theorem \ref{mildcomp2D}, all eigenspaces of~$\Ll(\AA, G)$ are infinite dimensional. 
\end{Exa}

\begin{Exa}[FAB example]
Following arguments given by \cite{Gras}, we enrich the example given in \cite[Part $2.1$]{split}, and obtain an example where~$G$ is FAB, i.e. every open subgroup has finite abelianization (for more details, see Example \ref{aritex}, and for references on FAB groups, see \cite{Labute} and \cite{MaireFAB}). 

Take~$p=3$, and consider~$\K:=\Q(\sqrt{-163})$. Then we define~$\Delta:=\Gal(\K/\Q)=\Z/2\Z$, and fix~$\chi_0$ the non-trivial irreducible character of~$\Delta$ over~$\F_p$. Consider the following set of places in $\Q$:~$\{7,19,13,31,337,43\}$. The class group of~$\K$ is trivial, the primes~$7,19,13,31,337$ are inert in~$\K$, and the prime~$43$ totally splits in~$\K$. 

Define~$S$ the primes above the previous set in~$\K$, and~$\K_S$ the maximal~$p$-extension unramified outside~$S$. Then~$\Delta$ acts on~$G:=\Gal(\K_S/\K)$, which is FAB by Class Field Theory. 

We can show that the pro-$p$ group~$G$ is mild, and Proposition \ref{numbtheo} gives:
$$gocha^\ast(\F_p,t):=\frac{1}{1-(6+\chi_0) t+ (6+\chi_0)t^2}, \quad \text{and} \quad gocha_{\chi_0}(\F_p,t):=\frac{1}{1-t-5t^2+6t^4}.$$
\\Therefore by Theorem \ref{mildcomp2D}, all eigenspaces of~$\Ll(\F_p, G)$ are infinite dimensional. 
\end{Exa}

\subsection*{Notations}
We follow the notations and definitions of \cite{AN0} and \cite[Appendice A]{LAZ}.

Let~$p$ be an odd prime, and~$G$ a finitely generated pro-$p$ group with minimal presentation~$1\to R \to F \to G \to 1$, and denote by~$\AA$ one of the rings~$\F_p$ or~$\Z_p$. Assume that~$Aut(G)$ contains a cyclic subgroup~$\Delta$ of order~$q$, a prime factor of~$p-1$. By \cite[Lemma~$2.15$]{hajir2019prime}, we observe that~$\Delta$ lifts to a subgroup of~$Aut(F)$.

When the context is clear, we omit the~$\AA$ when denoting filtrations (and associated invariants). Additionally, we always suppose that~$\Ll(\AA,G)$ is \textbf{torsion-free over}~$\AA$. 

Denote by~$Al(\AA, G)$ the completed group algebra of~$G$ over~$\AA$ and observe that~$G$ embeds naturally into~$Al(\AA, G)$.

For~$\chi\in {\rm Irr}(\Delta)$, we fix~$\{x_{j}^{\chi} \}_{1\leq j \leq d^{\chi} }$ a lift in~$F$ of a basis of~$\Ll_1(\AA,G)\lbrack\chi\rbrack$, where \linebreak~$d^{\chi}:=\rk_\AA \Ll_1(\AA,G)\lbrack\chi\rbrack$; by \cite[Corollaire~$3$, Proposition~$42$, Chapitre 14]{Ser}, this basis does not depend on the choice of~$\AA$. The Magnus isomorphism, from \cite[Chapitre II, Partie~$3$]{LAZ}, gives us the following identification of~$\AA$-algebras between~$Al(\AA, F)$ and the noncommuative series over~$X_j^{\chi}$'s with coefficients in~$\AA$: 
\begin{equation}\label{Magnus iso}
\phi_\AA \colon Al(\AA, F) \simeq \AA \langle \langle X^{\chi}_j; \chi \in {\rm Irr}(\Delta), 1\leq j \leq d^{\chi} \rangle\rangle; \quad x^{\chi}_j \mapsto X^{\chi}_j+1
\end{equation}

Define~$E(\AA)$ as the algebra~$\AA\langle\langle X^{\chi}_j; \chi\in {\rm Irr}(\Delta), 1\leq j \leq d^{\chi} \rangle \rangle$ filtered by~${\deg(X^{\chi}_j)=1}$ and write~$\{E_n(\AA)\}_{n\in \NN}$ for its filtration. One introduces~$I(\AA,R)$ the ideal of $E(\AA)$ generated by~$\{\phi_\AA(r-1); r\in R\}$ endowed with the induced filtration~$\{I_n(\AA,R):=I(\AA,R)\cap E_n(\AA)\}_{n\in \NN}$, and~$E(\AA, G)$ the quotient filtered algebra~$E(\AA)/I(\AA,R)$, with induced filtration~$\{E_n(\AA, G)\}_{n\in \NN}$.

We call~$M:=\bigoplus_{n\in \NN} M_n$ a graded locally finite ($\AA\lbrack \Delta \rbrack$-)module, if~$M_n$ is a finite dimensional ($\AA\lbrack \Delta\rbrack$)-module for every integer~$n$; and denote its Hilbert series by:
$$M(t):=\sum_{n\in \NN} (\rk_\AA M_n) t^n.$$

We make the following convention; we say that~$M$ is an~$\AA$-Lie algebra if~$M$ is a graded locally finite Lie algebra over~$\AA$, and when~$\AA:=\F_p$ we assume in addition that~$M$ is a restricted~$p$-Lie algebra. 
Recall the following graded locally finite~$\AA\lbrack \Delta\rbrack$-module and~$\AA$-Lie algebra, defined at the beginning:
\begin{equation*}
\begin{aligned}
\E(\AA):=\bigoplus_{n\in \NN} \E_n(\AA), \quad \text{where} \quad \E_n(\AA):=E_n(\AA)/E_{n+1}(\AA),
\\ \Ll(\AA,G):=\bigoplus_{n\in \NN} \Ll_n(\AA,G), \quad \text{and} \quad
\E(\AA, G):=\bigoplus_{n\in \NN} \E_n(\AA,G).
\end{aligned}
\end{equation*}

If~$P:=\sum_{n\in \NN} p_nt^n$ and~$Q:=\sum_{n\in \NN} q_n t^n$ are two series with real coefficients, we say that :~$P\leq Q \iff \forall n\in \NN, \quad p_n\leq q_n.$ We denote by~$\mu$ the Möbius function.

\section{An equivariant version of Min{\'{a}}{\v{c}}-Rogelstad-Tân's results} \label{rep}

Recall:

\begin{equation*}
gocha^\ast(\AA,t):=\sum_{n\in \NN} \left(\sum_{\chi \in {\rm Irr}(\Delta)} c_n^{\chi}{\chi}\right) t^n \in R\lbrack \Delta \rbrack\lbrack \lbrack t \rbrack \rbrack,
\end{equation*}
where $R\lbrack \Delta\rbrack$ denotes the finite representation ring of $\Delta$ (over $\AA$).

\subsection{Equivariant Hilbert series}
The aim of this subpart is to prove the following formula:
\begin{equation}\label{JeLaRep}
\begin{aligned}
gocha^\ast(\AA,t)=\prod_{n\in \NN} \prod_{\chi \in {\rm Irr}(\Delta)}P_\chi(\AA, t^n)^{a_n^{\chi}}, 
\\ \text{where} \quad P_\chi(\F_p,t):=\frac{1- \chi.t^p}{1-\chi.t}, \quad \text{and} \quad P_\chi(\Z_p,t):=\frac{1}{1-\chi.t}.
\end{aligned}
\end{equation}
This is Theorem \ref{Jen} defined in our introduction.

\begin{defi}
Let~$M:=\bigoplus_{n\in \NN} M_n$ be an~$\AA$-Lie algebra, graded locally finite~$\AA\lbrack \Delta \rbrack$-module, with basis~$\{x_{n,1};\dots;x_{n,m_n}\}_{n\in \NN}$, where~$m_n:=\rk_\AA M_n$.
We define:
\begin{itemize}
\item[$\bullet$] the graded locally finite module with basis given by words on~$\{x_{n,j}\}_{n\in \NN; j\in [\![1;m_n]\!]}$ by:
$$\tilde{U}(M):=\bigoplus_{n\in \NN} \tilde{U}(M)_n,$$
moreover, when~$\AA:=\F_p$, we also assume that the~$p$-restricted operation is compatible with the multiplicative structure of~$\tilde{U}(M)$;
\item[$\bullet$] the equivariant Hilbert series of~$M$ with coefficient in~$R\lbrack \Delta \rbrack$ by:
\begin{equation*}
\begin{aligned}
M^\ast(t)&:=\sum_{n\in \NN} \left( \sum_{\chi \in {\rm Irr}(\Delta)} m_n^{\chi}\chi \right) t^n &
\\ \text{where} \quad m_n^{\chi} &:=\rk_\AA M_n\lbrack \chi\rbrack \quad \text{ for every integer } n.&
\end{aligned}
\end{equation*}
\end{itemize}
\end{defi}

\begin{rema}
Since the action of~$\Delta$ over a graded locally finite module is semi-simple, it always preserves the grading. Consequently, if~$M$ is a graded locally finite~$\AA\lbrack \Delta \rbrack$-module, then the graded locally finite module~$\tilde{U}(M)$ is also endowed with a natural structure of graded locally finite~$\AA\lbrack \Delta \rbrack$-module. 
\end{rema}

We give a well-known result on Lie algebras, telling us that~$\tilde{U}$ is a universal enveloping algebra of~$M$.

\begin{theo}[Poincaré-Birkhoff-Witt]\label{PBW}
Let~$M$ be a graded locally finite~$\AA\lbrack\Delta \rbrack$-module and~$\AA$-Lie algebra. Then~$\tilde{U}(M)$ is a graded locally finite~$\AA\lbrack \Delta \rbrack$-module, universal~$\AA$-Lie algebra of~$M$.
\end{theo}

\begin{proof}
When~$\AA:=\Z_p$, see for instance \cite[Theorem~$2.1$]{Labute}.
\\When~$\AA:=\F_p$, see for instance \cite[Proposition~$12.4$]{DDMS}.
\end{proof}

\begin{coro}\label{proofFormmin}
The set~$\E(\AA,G)$ is a graded locally finite,~$\AA$-universal Lie algebra of~$\Ll(\AA,G)$. Consequently~$\E(\AA,G)$ is torsion-free.
\end{coro}

\begin{proof}
Let us first prove that~$\E(\AA,G)$ is a graded locally finite,~$\AA$-universal Lie algebra of~$\Ll(\AA,G)$. By Theorem \ref{PBW}, we only need to show that~$\tilde{U}(\Ll(\AA,G)) \simeq \E(\AA,G)$.

For~$\AA:=\F_p$, see \cite[Appendice A, Théorème~$2.6$]{LAZ}.

For~$\AA:=\Z_p$, the proof of \cite[Theorem~$1.3$]{HARTL20103276} carries to the case $E(\Z_p,G)$ with minor alterations. We consider $\Z_p$ and $\Q_p$ rather than $\Z$ and $\Q$. Furthermore, we conclude using the fact that $G$ is finitely generated, so $\grad(E(\Z_p,G))=\E(\Z_p,G)$ is isomorphic to $\grad(\Z_p[G])$, where $\Z_p[G]$ is filtered by power of the augmentation ideal over $\Z_p$. 
\end{proof}

\begin{rema}
Notice that~$\E(\AA,G)$ is also isomorphic to~$\tilde{U}(\Ll(\AA,G))$ as an~$\AA\lbrack \Delta\rbrack$-module. Therefore, we have:
$$\tilde{U}(\Ll(\AA,G))^\ast(t):=gocha^\ast(\AA,t).$$
\end{rema}

Before proving Formula \ref{JeLaRep}, we need the following result:

\begin{lemm}\label{Jerep}
Let~$M$ be a graded locally finite~$\AA\lbrack \Delta\rbrack$-module and~$\AA$-Lie algebra,
then:
\begin{equation*}
\begin{aligned}
\tilde{U}(M)^\ast(t)=\prod_{n\in \NN} \prod_{\chi \in {\rm Irr}(\Delta)} P_\chi(\AA,t^n)^{m_n^{\chi}}, 
\\ \text{ where } \quad P_\chi(\F_p,t):=\frac{1-\chi.t^p}{1-\chi.t}, \quad \text{and} \quad P_\chi(\Z_p,t):=\frac{1}{1-\chi.t}.
\end{aligned}
\end{equation*}
\end{lemm}

\begin{proof}
Let us first prove the case~$\AA:=\F_p$.

We are inspired by the proof of \cite[Corollary~$12.13$]{DDMS}. Observe that if~$M$ and~$N$ are graded locally finite~${\F_p}\lbrack \Delta \rbrack$-modules, then ~$M\bigotimes_{\F_p} N$ is also a graded locally finite~$\F_p\lbrack \Delta \rbrack$-module; moreover~$(M\bigotimes_{\F_p} N)^\ast(t):=M^\ast(t) N^\ast(t)$, and~$\tilde{U}(M\bigoplus N)=\tilde{U}(M)\bigotimes_{\F_p} \tilde{U}(N)$. So assume that :
$$M^\ast(t):=\sum_n m_n \chi_0 .t^n , \quad \text{ for some fixed and non-trivial } \chi_0 \in {\rm Irr}(\Delta).$$ 
Consider~$X_n:=\{x_{n,1},\dots, x_{n,m_n}\}$, an~${\F_p}\lbrack \Delta \rbrack$-basis of~$M_n$, where each~$x_{n,j}$ is of degree~$n$. Then a graded locally finite~${\F_p}\lbrack \Delta \rbrack$-basis of~$M$ is given by the (disjoint) union of all~$X_n$'s. 
Denote by~$$\tilde{U}(M)^\ast(t):=\sum_{r\in \NN} \left(\sum_{\chi \in {\rm Irr}(\Delta)}u_r^{\chi}\chi\right) t^r, \text{ where } u_r^{\chi}:=\dim_{\F_p} \tilde{U}(M)_r\lbrack \chi\rbrack.$$
We need to compute~$u_r^{\chi_0^i}$, where~$i\in \Z/q\Z$: this is the number of products of the form
$$\prod_{n=1}^r \prod_{j=1}^{m_n} (x_{n,j} \chi_0)^{m_{n,j}}, \quad \text{where} \quad 0 \leq m_{n,j}\leq p-1,$$
such that
$$\sum_{n=1}^r \sum_{j=1}^{m_n} nm_{n,j}=r \quad \text{ and } \quad \sum_{n=1}^r \sum_{j=1}^{m_n} m_{n,j}\equiv i \pmod q.$$
Notice that the coefficient before~$t^r$ of the polynomial
$$\prod_{n=1}^r{\lbrack 1+\chi_0 t^n+\dots+ \chi_0^{p-1} t^{(p-1)n}\rbrack^{m_n}}$$
is
\begin{equation*}
\sum_{n=1}^r \left( \sum_{j=1}^{m_n} {\chi_0}^{m_{n,j}} \right) t^r, \quad \text{where} \quad 0\leq m_{n,j} \leq p-1, \quad \text{and} \quad \sum_{n=1}^r \sum_{j=1}^{m_n} nm_{n,j}=r.
\end{equation*}

Consequently the coefficient before~$\chi_0^i t^r$ is exactly~$u_r^{\chi_0^i}$.
\medskip

Let us now prove the case~$\AA:=\Z_p$.
\\By the Poincaré-Birkhoff-Witt Theorem, the set~$\tilde{U}(M)$ is the symmetric Lie algebra over~$M$. Similarly to the previous case, we just need to study the case where there exists a unique~$\chi_0$ such that~$M^{\ast}(t):= \sum_n m_n \chi_0.t^n$. We get:
$$\tilde{U}(M)^\ast(t)=\prod_n \left(\frac{1}{1-\chi_0.t}\right)^{m_n^{\chi}}.$$
One deduces the general case.
\end{proof}
 
\begin{proof}[Proof of Formula \eqref{JeLaRep}]
We apply Lemma \ref{Jerep} and Corollary \ref{proofFormmin} to obtain:
\begin{equation*}
\begin{aligned}
gocha^\ast(\AA,t)=\prod_{n\in \NN} \prod_{\chi \in {\rm Irr}(\Delta)}P_\chi(\AA, t^n)^{a_n^{\chi}}, 
\\ \text{ where } \quad P_\chi(\F_p, t):=\frac{1-\chi.t^p}{1-\chi.t},\quad \text{and} \quad P_\chi(\Z_p, t):=\frac{1}{1-\chi.t}
\end{aligned}
\end{equation*}
\end{proof}

\subsection{Proof of Formula \eqref{ap}}

The aim of this part is to prove the following Proposition:
\begin{prop}\label{minfin}
Write~$n=mp^k$ with~$(m,p)=1$ and~$(n,q)=1$, then:
\begin{equation*}
\begin{aligned}
a_n^{\chi}(\F_p)=w_m^{\chi}(\F_p)+w_{mp}^{\chi}(\F_p)+\dots+w_{mp^k}^{\chi}(\F_p), \quad \text{and} \quad a_n^{\chi}(\Z_p)=w_n^{\chi}(\Z_p);
\\ \text{ where } \quad w_n^{\chi}:=\frac{1}{n}\sum_{m|n}\mu(n/m)mb_m^{\chi^{m/n}}\in \Q.
\end{aligned}
\end{equation*}
\end{prop}
This is Formula \eqref{ap} given in our introduction.

The strategy of the proof is to transform the product formula given by \eqref{JeLaRep}, into a sum in~$(R\lbrack \Delta \rbrack \otimes_\Z \Q) \lbrack \lbrack t \rbrack \rbrack$.

\begin{defi}[$\log$ function]
If~$P\in 1+tR\lbrack \Delta \rbrack \lbrack \lbrack t\rbrack \rbrack$, we define:
$$\log(P)(t):=-\sum_n \frac{(1-P(t))^n}{n}\in (R\lbrack \Delta \rbrack \otimes_\Z \Q)\lbrack \lbrack t \rbrack \rbrack.$$ 
\end{defi}

\begin{rema}
Note that the~$\log$ function enjoys the following properties:
\begin{enumerate}[\quad (i)]
\item If~$P$ and~$Q$ are in~$1+tR\lbrack \Delta \rbrack \lbrack \lbrack t\rbrack \rbrack$, then:
$$\log(PQ)=\log(P)+\log(Q),\quad \text{ and }$$
$$\log(1/P)=-\log(P).$$
\item If~$u$ is in~$tR\lbrack \Delta \rbrack \lbrack\lbrack t \rbrack\rbrack$, then 
$$\log\left(\frac{1}{1-u}\right)=\sum_{\nu=1}^\infty \frac{u^\nu}{\nu}.$$ 
\end{enumerate}
\end{rema}
Define the sequence~$(b_n^{\chi}(\AA))_{n\in \NN} \in \Q^{\NN}$ by:
\begin{equation*}
\log(gocha^\ast(\AA,t))=\sum_{n\geq 1} \left(\sum_{\chi \in {\rm Irr}(\Delta)} b_n^{\chi}(\AA) \chi\right) t^n.
\end{equation*}

\begin{prop}\label{bw}
If~$(n,q)=1$, we infer:
\begin{equation*}
b_n^{\chi^n}(\F_p):=\frac{1}{n}\left(\sum_{m|n}ma_m^{\chi^m}(\F_p)-\sum_{rp|n}rpa_r^{\chi^r}(\F_p)\right), \quad \text{and} \quad b_n^{\chi^n}(\Z_p):=\frac{1}{n} \sum_{m|n}ma_m^{\chi^m}(\Z_p).
\end{equation*}

\end{prop}

\begin{proof}
Let us just prove the case~$\AA:=\F_p$ (the case~$\AA:=\Z_p$ is similar).

First, Formula \eqref{JeLaRep} gives us:
$$gocha^\ast(\F_p, t)=\prod_{n\in \NN} \prod_{\chi \in {\rm Irr}(\Delta)}\left(\frac{1-\chi.t^{np}}{1-\chi.t^n}\right)^{a_n^{\chi}}.$$ 
Let us take the logarithm to obtain:
$$\log(gocha^\ast(\F_p, t))=\sum_n \sum_{\chi \in {\rm Irr}(\Delta)} a_n^{\chi}\left\lbrack \log(1-(\chi.t^n)^{p})-\log(1-\chi.t^n)\right\rbrack,$$
so that
$$\sum_{n\in \NN} \left(\sum_{\chi \in {\rm Irr}(\Delta)} b_n^{\chi} \chi\right)t^n=\sum_{w=1}^\infty \sum_{\chi \in {\rm Irr}(\Delta)} a_{w}^{\chi}\left( \sum_{\nu=1}^\infty \frac{(\chi.t^w)^{\nu}}{\nu}-\sum_{r=1}^\infty \frac{(\chi.t^w)^{rp}}{r}\right),$$
from which we conclude
$$\sum_{n=1}^\infty n \left(\sum_{\chi \in {\rm Irr}(\Delta)} b_n^{\chi} \chi\right) t^n = \sum_{n=1}^\infty \left( \sum_{\chi \in {\rm Irr}(\Delta)} (\sum_{m|n}ma_m^{\chi}\chi^{n/m}-\sum_{rp|n}rp a_{r}^{\chi}\chi^{n/r} )\right) t^n.$$
Then we infer:
$$nb_n^{\chi^n} = \sum_{m|n}ma_m^{\chi^m}-\sum_{rp|n}rp a_{r}^{\chi^r}.$$
\end{proof}

\begin{proof}[Proof of Proposition \ref{minfin}]
Again, we just prove the case~$\AA:=\F_p$.
\\We are inspired by the proof of \cite[Theorem~$2.9$]{Minac}.
\\First, we assume~$(n,p)=1$, then by Proposition \ref{bw}, we obtain:
$$nb_n^{\chi^n}=\sum_{m|n} ma_m^{\chi^m}.$$
So, using the Möbius inversion Formula, we obtain:
\begin{equation*}
a_n^{\chi^n}=w_n^{\chi^n}, \quad \text{ thus } \quad a_n^{\chi}=w_n^{\chi}.
\end{equation*}

Now, let us assume~$p$ divides~$n$. We show by induction on~$n$ that:
\begin{equation}\label{ind}
a_n^{\chi^n}=a_{n/p}^{\chi^{n/p}}+ w_n^{\chi^n} \tag{$\ast$}
\end{equation}
\begin{itemize}
\item[$\bullet$]
If~$n=p$, then by Proposition \ref{bw}, we have:~$pb_{p}^{\chi^p}=pa_{p}^{\chi^p}+a_{1}^{\chi}-pa_{1}^{\chi}$. So, 
$$pw_{p}^{\chi^p}=pb_{p}^{\chi^p}-b_{1}^{\chi}=pa_{p}^{\chi^p}-pa_{1}^{\chi}.$$
Therefore,~$a_{p}^{\chi^p}=a_{1}^{\chi}+w_{p}^{\chi^p}$.
\item[$\bullet$] 
Let us fix~$n$, an integer such that~$p|n$, and assume equation \eqref{ind} is true for all~$m$ such that~$m\neq n$ and~$p|m|n$. Then, following Proposition \ref{bw}, we have:
\begin{flalign*}
nb_n^{\chi^n} &= \sum_{m|n} ma_m^{\chi^m} -\sum_{rp|n} rpa_{r}^{\chi^r} &
\\ &= \sum_{m|n; (m,p)=1} ma_m^{\chi^m} + \sum_{p|m|n} m\left(a_m^{\chi^m}-a_{m/p}^{\chi^{m/p}}\right) &
\\ &=\sum_{m|n; (m,p)=1} mw_m^{\chi^m} +\sum_{p|m|n; m\neq n} mw_m^{\chi^m}+ n\left( a_n^{\chi^n}-a_{n/p}^{\chi^{n/p}}\right) &
\\ &= \sum_{m|n; m\neq n } mw_m^{\chi^m} + n\left( a_n^{\chi^n}-a_{n/p}^{\chi^{n/p}} \right). & 
\end{flalign*}
Moreover, by the Möbius inversion formula, we have:
$$nb_n^{\chi^n}=\sum_{m|n} mw_m^{\chi^m}.$$
Therefore, we obtain:
$$nw_n^{\chi^n}=n\left( a_n^{\chi^n} -a_{n/p}^{\chi^{n/p}} \right).$$
\end{itemize}

\end{proof}

\begin{rema}
Formula \eqref{ap} was already given for groups defined by one quadratic relation by Filip \cite[Formula~$(4.7)$]{Fi} (for~$\C$-representations in a geometrical context) and by Stix \cite[Formula~$(14.16)$]{stix2013rational} (in a Galois-theoretical context). Additionally, they computed explicitely the coefficients~$b_n^{\chi}(\Z_p)$. We discuss this analogy in Theorem \ref{Weigen}.
\end{rema}

\begin{rema}
Let us reformulate \cite[Question~$2.13$]{Minac}, asked by Min{\'a}{\v c}-Rogelstad-Tân, in our equivariant context:
\begin{quote}
\emph{Do we have for every integer~$n$ and every irreducible character~$\chi$, the equality~$c_n^\chi(\Z_p)=c_n^\chi(\F_p)?$}
\end{quote}

Later in this paper, we give a positive answer to this question, when~$G$ is finitely presented and~$\cd(G)\leq 2$ (see Theorem \ref{MRTsol}).
\end{rema}

\section{Infinite dimensional eigenspaces of~$\Ll(\AA, G)$} \label{assy}

The goal of this part is to study infinite dimensional eigenspaces (as a free~$\AA$-module) of~$$\Ll(\AA, G):=\bigoplus_{n\in \NN} \Ll_n(\AA,G),\quad \text{where} \quad \Ll_n(\AA,G):=G_n(\AA)/G_{n+1}(\AA).$$
For this purpose, we introduce~$\chi_0$-filtrations.

\subsection{Definition of~$\chi_0$-filtrations}\label{DfilLamb}

From now on, we make no distinction between~$\Z/q\Z$ and the set~$[\![ 1; q]\!]$. Observe the following isomorphism of groups, which depends on the choice of a fixed non-trivial irreducible character~$\chi_0$: 
$$\psi_{\chi_0}\colon ({\rm Irr}(\Delta);\otimes) \to (\Z/q\Z;+); \quad \chi_0^i \mapsto i.$$

Recall that~$\phi_\AA$ denotes the Magnus' isomorphism introduced in \eqref{Magnus iso}. We define~$E_{\chi_0}(\AA)$ as the~$\AA$-algebra~$\AA\langle\langle X^{\chi}_j ; \chi\in {\rm Irr}(\Delta), 1\leq j \leq d^{\chi} \rangle \rangle$ filtered by~$\deg(X^{\chi}_j)=\psi_{\chi_0}(\chi)$, and~$\{E_{\chi_0,n}(\AA)\}_{n\in \NN}$ as its filtration: called the~$\chi_0$-filtration of~$Al(\AA, F)$. We introduce
$$\E_{\chi_0}(\AA):=\bigoplus_{n\in \NN} \E_{\chi_0,n}(\AA), \quad \text{where} \quad \E_{\chi_0,n}(\AA):=E_{\chi_0,n}(\AA)/E_{\chi_0,n+1}(\AA).$$
Write~$I_{\chi_0}(\AA,R)$ for the two-sided ideal generated by~$\{\phi_\AA(r-1); r\in R\}\subset E_{\chi_0}(\AA)$, endowed with filtration~$\{I_{\chi_0,n}(\AA,R):=I_{\chi_0}(\AA,R)\cap E_{\chi_0,n}(\AA)\}_{n\in \NN}$; and~$E_{\chi_0}(\AA, G)$ the quotient filtered algebra~$E_{\chi_0}(\AA)/I_{\chi_0}(\AA,R)$. 
\\Define the following~$\AA$-module:
\begin{equation*}
\E_{\chi_0}(\AA,G):=\bigoplus_{n\in \NN}\E_{\chi_0,n}(\AA,G),\quad \text{where} \quad \E_{\chi_0,n}(\AA,G):=E_{\chi_0,n}(\AA,G)/E_{\chi_0,n+1}(\AA,G).
\end{equation*}
Introduce:
\begin{equation*}
\begin{aligned}
G_{\chi_0, n}(\AA):= \{g\in G; \phi_\AA(g-1)\in E_{\chi_0, n}(\AA, G)\}, \quad \text{and}
\\ \Ll_{\chi_0}(\AA,G):=\bigoplus_{n\in \NN} \Ll_{\chi_0,n}(\AA,G), \quad \text{where} \quad \Ll_{\chi_0,n}(\AA,G):= G_{\chi_0,n}(\AA)/G_{\chi_0,n+1 }(\AA).
\end{aligned}
\end{equation*}

We always assume that the $\AA$-Lie algebra $\Ll_{\chi_0}(\AA,G)$ is \textbf{torsion-free} over $\AA$. 

\begin{lemm}\label{reslie}
The set~$\E_{\chi_0}(\AA, G)$ is a graded locally finite,~$\AA$-universal Lie algebra of~$\Ll_{\chi_0}(\AA, G)$. Consequently, the graded~$\AA$-Lie algebra~$\E_{\chi_0}(\AA,G)$ is torsion-free.
\end{lemm}

\begin{proof}
This is similar to the proof of Corollary \ref{proofFormmin}.
\end{proof}

Since $G$ is finitely generated, we define:
\begin{equation*}
\begin{aligned}
gocha_{\chi_0}(\AA,t):= \sum_n c_{\chi_0, n}(\AA)t^n, \quad \text{where} \quad c_{\chi_0, n}(\AA):= \rk_\AA\E_{\chi_0,n}(\AA,G),
\\\quad \text{and} \quad a_{\chi_0, n}(\AA):= \rk_{\AA}(G_{\chi_0, n}(\AA)/G_{\chi_0, n+1}(\AA)).
\end{aligned}
\end{equation*}

\subsection{Properties of~$\chi_0$-filtrations}
This subpart aims to develop various properties of~$\chi_0$-filtations.

\begin{lemm}\label{delact}
The modules~$\E_{\chi_0}(\AA,G)$ and~$\Ll_{\chi_0}(\AA, G)$ are graded locally finite~$\AA\lbrack \Delta \rbrack$-modules. More precisely, we have:
\begin{equation*}
\begin{aligned}
\rk_\AA\E_{\chi_0, n}(\AA,G)\lbrack\chi\rbrack= & c_{\chi_0, n}(\AA) \delta_n^{\psi_{\chi_0}(\chi)},
\\ \rk_\AA\Ll_{\chi_0,n}(\AA,G)\lbrack\chi\rbrack= & a_{\chi_0, n}(\AA) \delta_n^{\psi_{\chi_0}(\chi)},
\end{aligned}
\end{equation*}

where~$\quad \delta_n^{\psi_{\chi_0}(\chi)}=1\quad$ if~$n\equiv \psi_{\chi_0}(\chi) \pmod q, \quad$ otherwise~$\quad \delta_n^{\psi_{\chi_0}(\chi)}=0$. 
\end{lemm}

\begin{proof}
Let us denote by~$\I_{\chi_0,n}(\AA,R):=I_{\chi_0,n}(\AA,R)/I_{\chi_0,n+1}(\AA,R)$.
Remind by \cite[Lemma~$2.15$]{hajir2019prime}, that~$\Delta\subset Aut(F)$ and~$\Delta(R)=R$.
So~$\E_{\chi_0}(\AA)$ is a graded locally finite~$\AA\lbrack \Delta\rbrack$-module, and~$\I_{\chi_0,n}(\AA,R)$ is stable by~$\Delta$.
By \cite[Chapitre I, Résultat~$2.3.8.2$]{LAZ}, we have the following exact sequence:
$$0\to \I_{\chi_0,n}(\AA,R)\to \E_{\chi_0,n}(\AA) \to  \E_{\chi_0,n}(\AA, G)\to 0.$$
Then~$\E_{\chi_0}(\AA, G)$ and~$\Ll_{\chi_0}(\AA, G)$ are graded locally finite~$\AA\lbrack \Delta \rbrack$-modules. Let us now study more precisely the~$\AA\lbrack \Delta \rbrack$-module structure of~$\E_{\chi_0}(\AA, G)$ and~$\Ll_{\chi_0}(\AA, G)$. 

For the structure of~$\E_{\chi_0}(\AA, G)$: take~$u\in \E_{\chi_0,n}(\AA)$ and write~$u:=X^{\chi_0^{i_1}}_{j_1}\dots X^{\chi_0^{i_u}}_{j_u},$ with~$i_1+\dots+i_u=n$. Therefore, for every~$\delta \in \Delta$,~$\delta(u)=\chi_0^n(\delta) u$. Then, we infer for every~$\chi$:
\begin{equation}\label{proofdelact}
\rk_\AA\E_{\chi_0,n}(\AA)\lbrack\chi\rbrack= \rk_\AA\E_{\chi_0,n}(\AA) \delta_n^{\psi_{\chi_0}(\chi)}. \tag{$\ast \ast$}
\end{equation}
Since~$\rk_\AA\E_{\chi_0,n}(\AA)\lbrack\chi\rbrack\geq \rk_\AA\E_{\chi_0,n}(\AA, G)\lbrack\chi\rbrack,$ we conclude by Equation \eqref{proofdelact} that:
$$\rk_\AA \E_{\chi_0,n}(\AA,G)\lbrack\chi\rbrack= c_{\chi_0, n} \delta_n^{\psi_{\chi_0}(\chi)}.$$

For the structure of~$\Ll_{\chi_0}(\AA, G)$: note by Lemma \ref{reslie} that~$\E_{\chi_0}(\AA, G)$ is a graded locally finite~$\AA\lbrack \Delta \rbrack$-module, universal~$\AA$-Lie algebra of~$\Ll_{\chi_0}(\AA, G)$. Hence for every~$\chi$, and every~$n$:
$$\rk_\AA\E_{\chi_0,n}(\AA, G)\lbrack\chi\rbrack\geq \rk_\AA \Ll_{\chi_0,n}(\AA,G)\lbrack\chi\rbrack.$$
This allows us to conclude for every~$\chi$:
$$\rk_\AA \Ll_{\chi_0,n}(\AA,G)\lbrack\chi\rbrack= a_{\chi_0, n} \delta_n^{\psi_{\chi_0}(\chi)}.$$
\end{proof}

Now, let us compare~$(c_{\chi_0, n})_{n\in \NN}$,~$(a_{\chi_0, n})_{n\in \NN}$,~$(c_n^{\chi})_{n\in \NN}$ and~$(a_n^{\chi})_{n\in \NN}$.
\begin{prop}\label{comp2}
The following inequalities hold:
\begin{equation}\label{comp21}
c_{\chi_0, qn+i}\leq \sum_{j=n}^{qn+i}c_j^{\chi_0^i}, \quad a_{\chi_0, qn+i}\leq \sum_{j=n}^{qn+i}a_j^{\chi_0^i},
\end{equation}

\begin{equation}\label{comp22}
c_n^{\chi}\leq  \sum_{k=\lceil{\frac{n-\psi_{\chi_0}(\chi)}{q}} \rceil}^{\lfloor{n-\psi_{\chi_0}(\chi)/q} \rfloor} c_{\chi_0, qk+\psi_{\chi_0}(\chi)}, \quad a_n^{\chi}\leq  \sum_{k=\lceil{\frac{n-\psi_{\chi_0}(\chi)}{q}} \rceil}^{\lfloor{n-\psi_{\chi_0}(\chi)/q} \rfloor} a_{\chi_0, qk+\psi_{\chi_0}(\chi)}.
\end{equation}
\end{prop}

\begin{proof}
Observe first that the~$\AA$-Lie algebras~$\Ll_{\chi_0}(\AA, G)$,~$\Ll(\AA,G)$,~$\E_{\chi_0}(\AA,G)$ and~$\E(\AA,G)$ are generated by~$\{X_j^\chi\}$. We only check inequalities involving~$c_n$ (proof of inequalities involving~$a_n$ are simlar).

Let us prove inequalities \eqref{comp21}. 
\\Take~$u$ in~$\E_{\chi_0,qn+i}(\AA,G)$. Since~$u$ is a sum of monomials~$u_l$ in~$\E_{\chi_0,qn+i}(\AA,G)$, we can assume that~$u$ is a monomial. So, let us write~$u=X^{\chi_0^{i_1}}_{j_1}\dots X^{\chi_0^{i_{r_u}}}_{j_{r_u}}$, where~$i_1+\dots+i_{r_u}=qn+i$. Consequently for every~$\delta \in \Delta$, 
\begin{equation*}
\begin{aligned}
\delta(u)& = {\chi_0}^{i_1}(\delta)X^{\chi_0^{i_1}}_{j_1}\dots {\chi_0}^{i_{r_u}}(\delta)X^{\chi_0^{i_{r_u}}}_{j_{r_u}}\quad \text{thus} & 
\\\delta(u)& = {\chi_0}^{i_1+\dots+i_{r_u}}(\delta) X^{\chi_0^{i_1}}_{j_1}\dots X^{\chi_0^{i_{r_u}}}_{j_{r_u}}=\chi_0^i(\delta)u.
\end{aligned}
\end{equation*}
Therefore~$u\in \E_{r_u}(\AA,G)\lbrack \chi_0^i \rbrack$. To conclude, we need to estimate~$r_u$. 
\begin{itemize}
\item[$\bullet$] If~$i_l=1$ for all~$l$, then~$r_u=qn+i$.
\item[$\bullet$] If~$i_l=q$ for all~$l$, then~$qr_u= qn+i$. Therefore,~$r_u\geq n$.
\end{itemize}
In any case:~$$n\leq r_u\leq qn+i.$$

Let us now prove inequalities \eqref{comp22}. 
\\Take~$u\in \E_n(\AA,G)\lbrack \chi \rbrack$. Since~$u$ is a sum of monomials, we can again assume that~$u$ is a monomial. Then by Lemma \eqref{delact}, we write~$u=X^{\chi_0^{i_1}}_{j_1}\dots X^{\chi_0^{i_n}}_{j_n}$, with~$i_1+\dots+i_n=kq+\psi_{\chi_0}(\chi)$ for some~$k$. Let us see which values can take~$k$:
\begin{itemize}
\item[$\bullet$] if each~$i_l=1$, one obtains~$kq+\psi_{\chi_0}(\chi)=n$, and so~$k\geq \lceil\frac{n-\psi_{\chi_0}(\chi)}{q}\rceil,$
\item[$\bullet$] if each~$i_l=q$, one obtains~$qn=kq+\psi_{\chi_0}(\chi)$, and so~$k\leq \lfloor \frac{qn-\psi_{\chi_0}(\chi)}{q}\rfloor$.
\end{itemize}
In any case:~$$\lceil{\frac{n-\psi_{\chi_0}(\chi)}{q}} \rceil \leq k \leq \lfloor{n-\psi_{\chi_0}(\chi)/q} \rfloor.$$
\end{proof}

\begin{rema}
Proposition \ref{comp2} was also given and proved by Anick: Proof of \cite[Theorem~$3$]{AN2}.
\end{rema}

\subsection{Some results on the series~$\log(gocha_{\chi_0}(\AA,t))$} 
In this subpart, we obtain information on~$(a_{\chi_0, n}(\AA))_{n\in \NN}$. For this purpose, we study the sequence~$(b_{\chi_0, n}(\AA))_{n\in \NN}$ namely defined by:
$$\log(gocha_{\chi_0}(\AA, t)):= \sum_{n\in \NN} b_{\chi_0, n} t^n.$$

\begin{theo}
The following equality holds in~$\NN\lbrack \lbrack t \rbrack \rbrack$:
\begin{equation*}
\begin{aligned}
gocha_{\chi_0}(\AA,t)=\prod_n P(\AA,t^n)^{a_{\chi_0, n}}, 
\\ \text{where} \quad P(\F_p,t):=\frac{1-t^{p}}{1-t}, \quad \text{and} \quad P(\Z_p,t):=\frac{1}{1-t}.
\end{aligned}
\end{equation*}
\end{theo}
 
\begin{proof}
By Lemma \ref{reslie},~$\E_{\chi_0}(\AA, G)$ is a graded locally finite,~$\AA$-universal Lie algebra of~$\Ll_{\chi_0}(\AA, G)$.  
\cite[Corollary~$2.2$]{RiSha} allows us to conclude the case~$\AA:=\F_p$, and \cite[Proposition~$2.5$]{Labute} allows us to conclude the case~$\AA:=\Z_p$.
\end{proof}

\begin{coro}\label{Delmin}
Let us write~$n=mp^k$, with~$(m,p)=1$, then:
\begin{equation*}
\begin{aligned}
a_{\chi_0,n}(\F_p)=\sum_{r=1}^k w_{\chi_0,mp^r}(\F_p), \quad \text{and} \quad a_{\chi_0, n}(\Z_p)=w_{\chi_0,n}(\Z_p); 
\\ \text{where} \quad w_{\chi_0,n}:=\frac{1}{n}\sum_{m|n} \mu(n/m)b_{\chi_0, m}.
\end{aligned}
\end{equation*}
\end{coro}

\begin{proof}
This proof is similar to the proof of \cite[Theorem~$2.9$]{Minac}.
\end{proof} 

\begin{coro}\label{cond}
The following assertions hold:
\begin{enumerate}[\quad (i)]
\item
If~$\chi$ is a non-trivial irreducible character, and there exists an infinite family of primes~$q_i\equiv \psi_{\chi_0}(\chi) \pmod q$ such that 
\begin{equation*}\label{chi-cond}
b_{\chi_0,q_i}>b_{\chi_0,1},
\end{equation*}
then~$\Ll(\AA, G)\lbrack \chi \rbrack$ is infinite dimensional.
\item 
If there exists an infinite family of primes~$(l_m)_m$ such that:
\begin{equation*}\label{1-cond}
b_{\chi_0,ql_m}\geq qb_{\chi_0,q}+l_mb_{\chi_0,l_m},
\end{equation*}
then~$\Ll(\AA, G)\lbrack \mathds{1} \rbrack$ is infinite dimensional.
\end{enumerate}
\end{coro}

\begin{proof}
This is a consequence of Corollary \ref{Delmin}.
\end{proof}

\begin{theo}\label{nonzero cond}
Assume there exist~$\alpha>1$ and a constant~$C\neq 0$ such that~$b_{\chi_0, n}\underset{n\to \infty}{\sim} C \alpha^n/n$.
Then every eigenspace of~$\Ll(\AA, G)$ is infinite dimensional.
\end{theo}

\begin{proof}
By Corollary \ref{cond}, we have: 
\begin{equation*}
\begin{aligned}
a_{\chi_0, q_i}=b_{\chi_0,q_i}-b_{\chi_0,1}/q_i, \quad \text{ and }  
\\a_{\chi_0, ql_m}=b_{\chi_0,ql_m}- qb_{\chi_0,q}-l_mb_{\chi_0,l_m}.
\end{aligned}
\end{equation*}

Since,~$b_{\chi_0,n}\underset{n\to \infty}{\sim}C\alpha^n/n$, we can find families of primes~$\{q_i\}_i$ and~$\{l_m\}_m$ where~$q_i$ and~$l_m$ are sufficiently big, such that:
$a_{\chi_0, q_i}>0$, and~$a_{\chi_0, ql_m}>0$.
Therefore by inequalities \eqref{comp21}, we extract an infinite subsequence of~$(a_n^{\chi})_n$ which is strictly positive. 
\end{proof}

\section{Examples}\label{expart}
Recall that~$1\to R \to F \to G \to 1$ denotes a minimal presentation of~$G$, and by \cite[Lemma~$2.15$]{hajir2019prime}, the group~$\Delta$ lifts to a subgroup of~$Aut(F)$. Keep in mind that~$\Ll(\AA,G)$ and~$\Ll_{\chi_0}(\AA,G)$ are assumed to be torsion-free over~$\AA$. Additionally here,~$G$ is assumed finitely presented, with cohomological dimension less than or equal to~$2$.

Consider the following~$\AA\lbrack \Delta\rbrack$-modules: 
$$R(\F_p):=R/R^p\lbrack R;F\rbrack,\quad \text{and}\quad R(\Z_p):=R/\lbrack R;F\rbrack.$$

Choose~$\chi_0$ a non-trivial element of~${\rm Irr}(\Delta)$. For every~$\chi \in {\rm Irr}(\Delta)$, we fix~$\{l^{\chi}_{j}\}_{1\leq j\leq r^{\chi}}$, where~$r^\chi:=\rk_\AA R(\AA)$, a lifting in ~$F$ of a basis of~$R(\AA)\lbrack\chi\rbrack$. By \cite[Corollaire~$3$, Proposition~$42$, Chapitre~$14$]{Ser}, these liftings do not depend on~$\AA$.

Recall that we defined, using the Magnus isomorphism~$\phi_\AA$ given by \eqref{Magnus iso}, the filtered algebras~$E(\AA, G)$ (in Notations) and~$E_{\chi_0}(\AA, G)$ (in Subpart \ref{DfilLamb}).
\\Name~$n^{\chi}_{j}$ (resp.~$n^{\chi}_{\chi_0, j}$) the least integer~$n$ such that~$\phi_\AA(l^{\chi}_{j}-1)$ is in~$E_n(\AA)\setminus E_{n+1}(\AA)$ (resp.~$E_{\chi_0, n}(\AA)\setminus E_{\chi_0, n+1}(\AA)$): this is the degree of~$l^{\chi}_{j}$ in~$E(\AA)$ (resp.~$E_{\chi_0}(\AA)$). We show in Lemma \ref{eulcomp} that these degrees do not depend on~$\AA$. Set the series:
\begin{equation*}
\begin{aligned}
\chi_{eul}(\AA,t) :=1-dt+\sum_{\chi;1\leq j \leq r^{\chi}} t^{n^{\chi}_{j}}, \quad 
\\ \chi_{eul}^\ast(\AA,t):=1-\sum_{\chi} d^{\chi}\chi.t+\sum_{\chi;1\leq j \leq r^{\chi}}\chi.t^{n^{\chi}_{j}},
\\ \chi_{eul,\chi_0}(\AA,t) :=1-\sum_{\chi} d^{\chi} t^{\psi_{\chi_0}(\chi)}+ \sum_{\chi;1\leq j \leq r^{\chi}} t^{n^{\chi}_{\chi_0, j}}.
\end{aligned}
\end{equation*}

\subsection{Generalities}
We give some generalities on groups of cohomological dimension less than or equal to~$2$.
\subsubsection{Computation of some~$gocha$ series}
Let us first recall the Lyndon's resolution, which allows us to compute~$gocha$ series as inverses of polynomials of the form~$\chi_{eul}$. A general reference is the article of Brumer \cite{BRU}. 

\begin{theo}\label{Lyn}
There exists a filtered~$E(\AA, G)$-module~$M(\AA)$, such that we have the following exact sequence of filtered~$E(\AA, G)$-modules:
\begin{multline*}
0\to M(\AA) \to  \bigoplus_{\chi;1\leq j \leq r^{\chi}} (\phi_\AA(l^{\chi}_j-1)) E(\AA, G) \to 
\\ \bigoplus_{\chi;1\leq j \leq d^{\chi}} (\phi_\AA(x^{\chi}_j-1))E(\AA, G) \to E(\AA, G) \to \AA \to 0.
\end{multline*}
And the cohomological dimension of~$G$ is less than or equal to two if and only if~$M(\AA)=0$.

There exists a filtered~$E_{\chi_0}(\AA, G)$-module~$M_{\chi_0}(\AA)$, such that we have the following exact sequence of filtered~$E_{\chi_0}(\AA, G)$-modules:
\begin{multline*}
0\to M_{\chi_0}(\AA) \to  \bigoplus_{\chi;1\leq j \leq r^{\chi}} (\phi_\AA(l^{\chi}_j-1)) E_{\chi_0}(\AA, G) \to 
\\ \bigoplus_{\chi;1\leq j \leq d^{\chi}} (\phi_\AA(x^{\chi}_j-1))E_{\chi_0}(\AA, G) \to E_{\chi_0}(\AA, G) \to \AA \to 0.
\end{multline*}
And the cohomological dimension of~$G$ is less than or equal to two if and only if~$M_{\chi_0}(\AA)=0$.
\end{theo}

\begin{rema}
Theorem \ref{Lyn} is true for every filtration over~$Al(\AA, G)$.
\end{rema}

\begin{proof}
Let us define the following $\AA$-modules:
$$K(\F_p):=R/R^p[R;R], \quad \text{and} \quad K(\Z_p):=R/[R;R].$$
Notice that $Al(\AA,G)$ acts on $K(\AA)$ via conjugation (see \cite[Part $7.3$]{Koch}). 

By \cite[Sequence $(5.2.2)$]{BRU}, we have the following sequence of~$Al(\AA, G)$-pseudocompact-modules:
$$0\to K(\AA) \to \bigoplus_{\chi,j} \phi_\AA(x_j^{\chi}-1)Al(\AA, G) \to Al(\AA, G) \to \AA \to 0.$$ 
By \cite[Theorem~$7.7$]{Koch}, there exists a~$Al(\AA, G)$-pseudocompact-module~$K'(\AA)$ such that we have the exact sequence:
$$0\to K'(\AA) \to \bigoplus_{\chi, j} \phi_\AA(l_j^{\chi}-1) Al(\AA, G) \to K(\AA) \to 0.$$
Furthermore~$\cd(G)\leq 2$ if and only if~$K'(\AA)=0$.
Therefore, we obtain the following resolutions:
\begin{multline*}
0\to M(\AA) \to  \bigoplus_{\chi;1\leq j \leq r^{\chi}} (\phi_\AA(l^{\chi}_j-1)) E(\AA, G) \to 
\\ \bigoplus_{\chi;1\leq j \leq d^{\chi}} (\phi_\AA(x^{\chi}_j-1))E(\AA, G) \to E(\AA, G) \to \AA \to 0,
\end{multline*}
where~$M(\AA)$ is the set~$K'(\AA)$ endowed with its structure of filtered~$E(\AA, G)$-module, and
\begin{multline*}
0\to M_{\chi_0}(\AA) \to  \bigoplus_{\chi;1\leq j \leq r^{\chi}} (\phi_\AA(l^{\chi}_j-1)) E_{\chi_0}(\AA, G) \to 
\\ \bigoplus_{\chi;1\leq j \leq d^{\chi}} (\phi_\AA(x^{\chi}_j-1))E_{\chi_0}(\AA, G) \to E_{\chi_0}(\AA, G) \to \AA \to 0,
\end{multline*}
where~$M_{\chi_0}(\AA)$ is the set~$K'(\AA)$ endowed with its structure of filtered~$E_{\chi_0}(\AA, G)$-module. 

Finally,~$\cd(G)\leq 2 \iff K'(\AA)=0 \iff M(\AA)=0 \iff M_{\chi_0}(\AA)=0.$ 
\end{proof}

Let us now compute~$gocha$ series:

\begin{prop} \label{key}
We have the following equivalences:
\begin{multline*}
\cd(G)\leq 2 \iff gocha(\AA,t)=\frac{1}{\chi_{eul}(\AA,t)} \iff 
\\ gocha^\ast(\AA,t)=\frac{1}{\chi_{eul}^\ast(\AA, t)} \iff gocha_{\chi_0}(\AA,t)=\frac{1}{\chi_{eul,\chi_0}(\AA,t)}.
\end{multline*}
\end{prop}

\begin{proof}
One denotes by~$\rho^{\chi}_{j}$ (resp.~$\rho^{\chi}_{\chi_0, j}$) the image of~$\phi_\AA(l^{\chi}_j-1)$ in~$\E_{n^{\chi}_{j}}(\AA)$ (resp.~$\E_{n^{\chi}_{\chi_0, j}}(\AA)$).

By \cite[Chapitre I, Formule~$2.3.8.2$]{LAZ}, the functor~$\grad$ is exact, then we apply \cite[Chapitre II, Proposition~$3.1.3$]{LAZ} and Theorem \ref{Lyn}, to obtain the following exact sequences of graded locally finite modules:
\begin{multline}\label{proofseq}
0\to \grad(M(\AA)) \to \bigoplus_{\chi;j} \rho^{\chi}_{j} \E(\AA,G) \to 
\\ \bigoplus_{\chi;j} X^{\chi}_j\E(\AA,G) \to \E(\AA,G) \to \AA \to 0, \tag{$\star$}
\end{multline}
\begin{multline}\label{proofchiseq}
0\to \grad(M_{\chi_0}(\AA)) \to \bigoplus_{\chi;j} \rho^{\chi}_{\chi_0,j} \E_{\chi_0}(\AA, G) \to 
\\ \bigoplus_{\chi;j} X^{\chi}_j\E_{\chi_0}(\AA, G) \to \E_{\chi_0}(\AA, G) \to \AA \to 0. \tag{$\star \star$}
\end{multline}

From Theorem \ref{Lyn} and sequence \eqref{proofseq}, we infer:
$$\cd(G)\leq 2 \iff M(\AA)=0 \iff \grad(M(\AA))=0 \iff gocha(\AA,t)=\frac{1}{\chi_{eul}(\AA,t)}.$$
Moreover Theorem \ref{Lyn} and sequence \eqref{proofchiseq} give us:
\begin{multline}
\cd(G)\leq 2 \iff M_{\chi_0}(\AA)=0 \iff \grad(M_{\chi_0}(\AA))=0 
\\ \iff gocha_{\chi_0}(\AA,t)=\frac{1}{\chi_{eul,\chi_0}(\AA,t)}.
\end{multline}

From the choice of the families $\{x_j^{\chi}\}$ and $\{\rho_j^\chi\}$, we infer that the sequence \eqref{proofseq} is exact in the category of graded locally finite~$\AA\lbrack \Delta \rbrack$-modules. This allows us to conclude: 
$$\cd(G)\leq 2 \iff gocha^\ast(\AA,t)=\frac{1}{\chi_{eul}^\ast(\AA,t)}.$$
\end{proof}

\subsubsection{Answer to \cite[Question~$2.13$]{Minac}}
Extending and reformulating \cite[Question~$2.13$]{Minac} in our equivariant context, when $G$ is finitely presented and $\cd(G)\leq 2$, we show in this Subsubpart that:
\begin{quote}
\emph{The series~$gocha(\AA,t)$,~$gocha^\ast(\AA,t)$ and~$gocha_{\chi_0}(\AA,t)$ do not depend on the ring~$\AA$?}
\end{quote}

\begin{lemm}\label{eulcomp}
Assume that~$\Ll(\Z_p,G)$ is torsion-free. Then, for every~$j$ and every~$\chi$, the integers~$n_j^{\chi}(\AA)$ do not depend on~$\AA$. Similarly, if~$\Ll_{\chi_0}(\Z_p,G)$ is torsion-free, then the integers~$n_{\chi_0,j}^{\chi}(\AA)$ do not depend on~$\AA$
\end{lemm}

\begin{proof}
Let us prove that~$n_j^{\chi}$ does not depend on~$\AA$. Recall that~$n_j^{\chi}(\F_p)$ (resp.~$n_j^{\chi}(\Z_p)$) is the degree of~$l^{\chi}_{j}$ in~$E(\F_p)$ (resp.~$E(\Z_p)$), and~$\rho^{\chi}_{j}(\F_p)$ (resp.~$\rho^{\chi}_{j}(\Z_p)$) denotes the image of~$\phi_{\F_p}(l_j^{\chi}-1)$ in~$\E_{n_j^{\chi}}(\F_p)$ (resp.~$\phi_{\Z_p}(l_j^{\chi}-1)$ in~$\E_{n_j^{\chi}}(\Z_p)$). Notice that we have a filtered surjection:
$$E(\Z_p)\overset{\pmod p}{\to} E(\F_p),\quad  \text{with kernel} \quad pE(\Z_p).$$
Since the choice of the family~$\{l_j^\chi\}_{j,\chi}$ does not depend on $\AA$, we infer that~$\phi_{\Z_p}(l_j^{\chi}-1)\equiv \phi_{\F_p}(l_j^\chi-1)\pmod{p}$.
Therefore,~$n_j^{\chi}(\Z_p)\leq n_j^{\chi}(\F_p)$. 

To show that~$n_j^{\chi}(\Z_p)=n_j^{\chi}(\F_p)$, it is sufficient to show that for every integer $j$, and character $\chi$, we have~$\rho_j^{\chi}(\Z_p)$ not in $p\E(\Z_p)$. 

From \cite[Proposition $4.3$]{FOR}, we infer the following isomorphism of $E(\Z_p,G)$-modules:
$$K(\Z_p):=R/[R;R]\simeq I(\Z_p,R)/E_1(\Z_p)I(\Z_p,R).$$
Since, $G$ is of cohomological dimension $2$, by \cite[Theorem $7.7$]{Koch}, we have $$K(\Z_p)\simeq \prod_{j,\chi}\phi_{\Z_p}(l_j^\chi-1)E(\Z_p,G).$$
Introduce
$$\I_n(\Z_p,R):=I_n(\Z_p,R)/I_{n+1}(\Z_p,R), \quad \text{and} \quad \I(\Z_p,R):=\bigoplus_{n\in \NN}\I_n(\Z_p,R).$$ Then, we observe that 
\begin{multline*}
\grad(E_1(\Z_p)I(\Z_p,R))=\grad(\prod_{i,\chi}X_i^\chi E(\Z_p)I(\Z_p,R))=\bigoplus_{i,\chi}\grad(X_i^\chi I(\Z_p,R))
\\ =\bigoplus_{i,\chi}X_i^\chi \I(\Z_p,R)=\E_1(\Z_p)\I(\Z_p,R).
\end{multline*}
Consequently
$$\grad(K(\Z_p))\simeq \bigoplus_{j,\chi}\rho_j^\chi(\Z_p)\E(\Z_p,G)\simeq \I(\Z_p,R)/\E_1(\Z_p)\I(\Z_p,R).$$

Assume now, by contradiction, that there exists one integer $j_0$ and one character $\chi_0$ such that $\rho_{j_0}^{\chi_0}(\Z_p)$ is in $p\E(\Z_p)$, then there exists $u\in \E(\Z_p)$ such that $\rho_{j_0}^{\chi_0}:=pu$.  Moreover, since $\E(\Z_p,G)$ is torsion-free, we deduce that $u$ is in $\I(\Z_p,R)$.  
Therefore, there exist elements $g_j^\chi$ in~$\E(\Z_p,G)$ such that $u\equiv \sum_{j,\chi} g_j^\chi \rho_j^\chi \pmod{\E_1(\Z_p)\I(\Z_p,R)}$. Consequently:
$$\rho_{j_0}^{\chi_0}:=pu \equiv \sum_{j,\chi} pg_j^\chi \rho_j^\chi \pmod{\E_1(\Z_p)\I(\Z_p,R)}.$$

Since the family $\rho_j^\chi$ is a basis of the free $\E(\Z_p,G)$-module~$\I(\Z_p,R)/\E_1(\Z_p)\I(\Z_p,R)$, we infer $pg_{j_0}^{\chi_0}=1$. This is impossible since $p$ is not invertible in $\E(\Z_p,G)$. 

\end{proof}

\begin{theo}\label{MRTsol}
Assume that~$\Ll(\Z_p,G)$ is torsion-free, then :
$$gocha(\Z_p, t)=gocha(\F_p, t), \quad \text{and} \quad gocha^\ast(\Z_p, t)=gocha^\ast(\F_p, t).$$
Furthermore, if $\Ll_{\chi_0}(\Z_p,G)$ is torsion-free, then
$$gocha_{\chi_0}(\Z_p, t)=gocha_{\chi_0}(\F_p, t).$$
\end{theo}

\begin{proof}
We apply Proposition \ref{key} and Lemma \ref{eulcomp}. 
\end{proof}

\begin{rema}
If we remove the hypothesis that~$Aut(G)$ contains a subgroup~$\Delta$ of order~$q$, then we still have:
$$gocha(\Z_p,t)=gocha(\F_p,t).$$
A criterion to obtain finitely presented groups of cohomological dimension less than or equal to~$2$ is given by \cite{Labute} when~$p$ is odd, and by \cite{Labute-Minac} when~$p=2$.
\end{rema}

\subsubsection{Gocha's series and eigenvalues}
Thanks to Proposition \ref{key}, we can compute~$gocha$ series. Then applying Formulae \eqref{minfor} and \eqref{ap}, we obtain an explicit equation relating coefficients~$a_n$ and~$a_n^{\chi}$. However, the computation of~$b_n$ has complexity~$n$ (more precisely it depends on~$\{c_m\}_{m\leq n}$).

If we consider roots of~$\chi_{eul}$, we infer a formula for~$b_n$ which depends on the arithmetic complexity of~$n$. The following results are mostly adapted in our context from ideas of Labute (\cite[Formula~$(1)$]{LABUTE197016}) and Weigel (\cite[Theorem D]{weigel2015graded}).

Let~$\deg(G)$ be the degree of~$\chi_{eul}$, and~$\lambda_i$ the eigenvalues of~$G$, written as:
$$\chi_{eul}(t):=\prod_{i=1}^{\deg(G)}(1-\lambda_it).$$
One denotes by~$M_n$ the necklace polynomial of degree~$n$:
$$M_n(t):=\sum_{m|n}\mu(n/m)\frac{t^m}{n}.$$

Let us state \cite[Theorem D]{weigel2015graded}:
\begin{theo}
Assume~$(n,q)=1$ and write~$n=mp^k$, with~$(m,p)=1$. Then we infer:
$$a_n(\Z_p)=\sum_{i=1}^n M_n(\lambda_i), \quad a_n(\F_p)=\sum_{i=1}^n\sum_{j=0}^k M_{mp^j}(\lambda_i).$$
\end{theo}

\begin{proof}
Weigel showed in the proof of \cite[Theorem~$3.4$]{weigel2015graded}, that:
$$\sum_{i=1}^n M_n(\lambda_i)=w_n.$$
Then we conclude using Theorem \ref{MRTsol} and Formula \eqref{minfor}.
\end{proof}

Let us adapt this result in an equivariant context. By a choice of a primitive~$q$-th root of unity, we have~$\F_q\subset \F_p^\times\subset \overline{\F_p}$, the algebraic closure of $\F_p$. Consider~$\delta$ a non-trivial element in~$\Delta$, and evaluate~$\chi_{eul}^\ast$ in~$\delta$ by:
$$\chi_{eul}^\ast(\delta)(t):=1-\sum_\chi c_1^{\chi}\chi(\delta)t+\sum_{\chi;1\leq j \leq r^{\chi}} \chi(\delta)t^{n_j^{\chi}}\in \F_p\lbrack t \rbrack \subset \overline{\F_p} \lbrack t \rbrack.$$
Define~$\{\lambda_{\delta,j}\}_{1\leq j \leq \deg(G)}\subset \overline{\F_p}$ the eigenvalues of~$\chi_{eul}^\ast(\delta)(t)$. We introduce~$\Ff(\Delta,\overline{\F_p})$ the~$\overline{\F_p}$-algebra of functions from~$\Delta$ to~$\overline{\F_p}$ and:
$$\eta_j\colon \Delta \to \overline{\F_p}; \quad \delta \mapsto \lambda_{\delta,j} .$$
Therefore, we infer:
$$\chi_{eul}^\ast(t):=\prod_{j=1}^{\deg(G)}(1-\eta_j t)\in \Ff(\Delta,\overline{\F_p})\lbrack t \rbrack.$$
Consequenlty, if we apply the $\log$ function to the previous equality, we obtain:
$$b_m^\ast:=\sum_{\chi \in {\rm Irr}(\Delta)} b_m^\chi \chi=\frac{\eta_1^m+\dots+\eta_{\deg(G)}^m}{m}.$$

Let us define for every $\eta\in \Ff(\Delta,\overline{\F_p})$:
$$M_n^\ast(\eta):=\sum_{m|n}\frac{1}{n}\mu(n/m)\eta^{m,(n/m)},\quad \text{where}\quad \eta^{m,(u)}(\delta)=\eta(\delta^u)^m.$$

\begin{prop}\label{Weigen}
Let us assume~$q$ divides~$p-1$ and $(n,q)=1$. Write~$n=mp^k$, with~$(m,p)=1$, then:
\begin{equation*}
\begin{aligned}
a_n(\Z_p)^\ast:=\sum_\chi a_n^{\chi}(\Z_p)\chi=\sum_{j=1}^{\deg(G)} M_n^{\ast}(\eta_j), \quad \text{and}
\\ a_n(\F_p)^\ast:= \sum_\chi a_n^{\chi}(\F_p)\chi=\sum_{j=1}^{\deg(G)}\sum_{i=0}^k M^\ast_{mp^i}(\eta_j),
\end{aligned}
\end{equation*}
the equality is in the~$\overline{\F_p}$-algebra~$\Ff(\Delta,\overline{\F_p})$.
\end{prop}

\begin{proof}
Let us remind that $b_n^\ast:=\sum_{\chi}b_n^{\chi}\chi$. After making the following change of variable: $\gamma=\chi^{n/m}$, we observe that for every $\delta$ in $\Delta$, we have
$$b_m^{\ast 1,(n/m)}(\delta):=b_m^\ast(\delta^{n/m})=\sum_\chi b_m^\chi\chi(\delta^{n/m})=\sum_{\chi\in {\rm Irr}(\Delta)} b_m^{\chi}\chi(\delta)^{n/m}=\sum_{\gamma \in {\rm Irr}(\Delta)} b_m^{\gamma^{m/n}}\gamma(\delta).$$
Consequently, $b_m^{\ast 1,(n/m)}=\sum_{\chi}b_m^{\chi^{m/n}}\chi.$
Since $mb_m^\ast=\eta_1^m+\dots+\eta_{\deg(G)}^m,$ we obtain:
$$mb_m^{\ast 1,(m/n)}=\left(\eta_1^m+\dots+\eta_{\deg(G)}^m\right)^{(n/m)}=\eta_1^{m,(n/m)}+\dots+\eta_{\deg(G)}^{m,(n/m)}.$$
Using Formula \eqref{ap}, the conclusion follows.
\end{proof}

\begin{rema}
Filip (\cite[Formula~$(4.8)$]{Fi}) and Stix (\cite[Formula~$(14.16)$]{stix2013rational}) also obtained Proposition \ref{Weigen} for some groups defined by one quadratic relation. They computed explicitely the functions~$\eta_j$.
\end{rema}

\begin{exem}
Let us illustrate Proposition \ref{Weigen}, with Example \ref{mildex}.

When splitting $\chi_{eul}^\ast$ into eigenvalues, we obtain:
$$\chi_{eul}^\ast(t)=(1-\eta_1t)(1-\eta_2t)=1-(\chi_0+\chi_0^2+\chi_0^3)t+\chi_0^3t^3,$$
Moreover, $\eta_1\eta_2=\chi_0^3$ and $\eta_1+\eta_2=\chi_0+\chi_0^2+\chi_0^3$ (as functions). Therefore, if we apply Proposition \ref{Weigen}, we get:
\begin{equation*}
\begin{aligned}
a_2^\ast:=\sum_\chi a_2^\chi=\frac{\eta_1^2+\eta_2^2-\eta_1^{(2)}-\eta_2^{(2)}}{2}=\frac{(\eta_1+\eta_2)^2-2\eta_1\eta_2-(\eta_1+\eta_2)^{(2)}}{2}
\\=\frac{\chi_0^2+\chi_0^4+\chi_0^6+2\chi_0^3+2\chi_0^4+2\chi_0^5-2\chi_0^3-\chi_0^2-\chi_0^4-\chi_0^6}{2}=\chi_0^4+\chi_0^5.
\end{aligned}
\end{equation*}

Let us now compute $a_3^\chi$. For this purpose, we first observe that 
\begin{equation*}
\begin{aligned}
\eta_1^3+\eta_2^3=(\chi_0+\chi_0^2+\chi_0^3)^3-3(\chi_0+\chi_0^2+\chi_0^3)\chi_0^3
\\=\chi_0^9+3\chi_0^8+6\chi_0^7+4\chi_0^6+3\chi_0^5+\chi_0^3.
\end{aligned}
\end{equation*}
Therefore, we have:
$$a_3^\ast:=\sum_\chi a_3^\chi=\frac{\eta_1^3+\eta_2^3-\eta_1^{(3)}-\eta_2^{(3)}}{3}=\frac{\eta_1^3+\eta_2^3-(\eta_1+\eta_2)^{(3)}}{3}=\chi_0^5+\chi_0^6+2\chi_0^7+\chi_0^8.$$
\end{exem}

Let us conclude this subpart by proving Theorem \ref{mildcomp2D} given in our introduction.
\begin{theo}\label{mildcomp2}
Assume that~$\Ll(\AA, G)$ is infinite dimensional and for some~$\chi_0$ that~$L_{\chi_0}(G)$ is reached for a unique eigenvalue~$\lambda_{\chi_0}$ such that:
\begin{enumerate}[\quad (i)]
\item~$\lambda_{\chi_0}$ is real,
\item~$L_{\chi_0}(G)=\lambda_{\chi_0}>1$.
\end{enumerate}
Then every eigenspace of~$\Ll(\AA, G)$ is infinite dimensional.
\end{theo}

\begin{proof}
We study the asymptotic behaviour of~$(b_{\chi_0, n})_{n\in \NN}$. By Proposition \ref{key}, we have: 
$$gocha_{\chi_0}(t):=\frac{1}{\chi_{eul, \chi_0}(t)}.$$
Let us denote by~$\{\lambda_1;\dots;\lambda_u\}$ the real~$\chi_0$-eigenvalues of~$G$ and~$\{\beta_1 e^{i\pm\theta_1};\dots; \beta_v e^{i\pm\theta_v}\}$ the polar forms of non real~$\chi_0$-eigenvalues of~$G$. Without loss of generality, assume that~$\lambda_{\chi_0}:=\lambda_1$. Let us write 
$$\chi_{eul, \chi_0}(t):=\prod_{i=1}^u(1-\lambda_it)\prod_{j=1}^v(1-\beta_je^{i\theta_j}t)(1-\beta_je^{-i\theta_j}t).$$
Then, we obtain:
$$\log(\chi_{eul, \chi_0}(t))=\sum_{n\in \NN} \frac{\sum_{i=1}^u \lambda_i^n+\sum_{j=1}^v \beta_j^n(e^{in\theta_j}+e^{-in\theta_j})}{n}t^n.$$
Thus~$b_{\chi_0, n}\underset{n\to \infty}{\sim} C\lambda_1^n/n$, for some~$C>0$. We conclude by Theorem \ref{nonzero cond}.
\end{proof}

\subsection{Group Theoretical examples} \label{grexa}

\subsubsection{Free pro-$p$ groups}
In this subpart, assume that~$G$ is a free finitely generated pro-$p$ group. Observe that
$\Ll(\Z_p,G)$ and $\Ll_{\chi_0}(\Z_p,G)$ are torsion-free.

\begin{theo}\label{freecomp2}
Assume that~$G$ is a noncommutative free pro-$p$ group, then every eigenspace of~$\Ll(\AA, G)$ is infinite dimensional.
\end{theo}

\begin{proof}
Let us fix a non-trivial character~$\chi_0\in {\rm Irr}(\Delta)$, such that~$d^{\chi_0}\leq d^{\chi}$ for every non-trivial~$\chi$. Then we have~$\chi_{eul,\chi_0}(t):=1-\sum_{i=1}^{q}d^{\chi_0^i}t^i$. Set~$s$ a minimal positive real root of~$\chi_{eul,\chi_0}$. We will show that~$s$ is the unique root of minimal absolute value of~$\chi_{eul,\chi_0}$.

We have:
$$0=1-\sum_{i=1}^qd^{\chi_0^i}s^i\leq 1-d^{\chi_0}s\sum_{i=0}^{q-2}s^i-d^{\mathds{1}}s^q\leq 1-d^{\chi_0}s-d^{\mathds{1}}s^q.$$
Then~$d^{\chi_0}s+d^{\mathds{1}}s^q\leq 1$. Thus~$s\leq \min\{1/d^{\chi_0}; (1/d^{\mathds{1}})^{1/q}\}$, so~$0<s<1$. 

If we denote by~$z$ a complex root (not in~$\rbrack 0; 1\lbrack$) of~$\chi_{eul,\chi_0}$, then we notice by the triangle inequality, that~$\chi_{eul,\chi_0}(|z|)<\chi_{eul,\chi_0}(z)=0$. Therefore~$|z|>s$.

Consequently,~$\chi_{eul,\chi_0}$ admits a unique root~$s$ of minimal absolute value which is in~$\rbrack 0;1 \lbrack$. Therefore, by Theorem \ref{mildcomp2D}, we conclude.
\end{proof}

Let us give some examples.

\begin{exem}
Consider~$\Delta:=\Z/2\Z$, and fix~$\chi_0$ the non-trivial irreducible character of~$\Delta$ over~$\AA$. Assume that~$G$ is a free pro-$p$ group with two generators~$\{x,y\}$, and~$\Delta$ acts on~$G$ by:~$\delta(x)=x$,~$\delta(y)=y^{-1}$. Then following our notations, we have:~$x=x^{\mathds{1}}$, and~$y=x^{\chi_0}$. Observe that~$Al(\AA,G)$ is a free algebra on two variables over~$\AA$.
\\Let us first compute some coefficients~$a_n^{\chi}$, with Formula \eqref{ap}. We have:
\begin{equation*}
gocha^\ast(\AA,t):=\frac{1}{1-(1+{\chi_0}). t},\quad \text{ and } \quad \log(gocha^\ast(\AA,t)):=\sum_n \frac{(1+{\chi_0})^n}{n} t^n.
\end{equation*}
So 
\begin{equation*}
\begin{aligned}
c_{2n}^{\mathds{1}}=c_{2n}^{\chi_0}=2^{2n-1}, \quad & c_{2n+1}^{\chi_0}=c_{2n+1}^{\mathds{1}}=2^{2n},
\\ b_{2n+1}^{\chi_0}=b_{2n+1}^{\chi_0}=\frac{2^{2n}}{2n+1},\quad \text{ and } \quad & b_{2n}^{\mathds{1}}=b_{2n}^{\chi_0}=\frac{2^{2n-1}}{2n}.
\end{aligned}
\end{equation*}
Assume for instance~$p\neq 3$, then one obtains:
\begin{equation*}
a_3^{\chi_0}=\frac{2^2-1}{3}=1, \quad \text{ and } \quad a_3^{\mathds{1}}=1.
\end{equation*}
\\Observe by Theorem \ref{freecomp2}, that every eigenspace of~$\Ll(\AA, G)$ is infinite.
\end{exem}

\begin{exem}
Again, take~$\Delta:=\Z/2\Z$ and~$\chi_0$ the unique non-trivial~$\AA$-irreducible character of~$\Delta$. Assume~$G$ is free generated by~$\{x^{\chi_0}_{1};\dots;x^{\chi_0}_{d}\}$.

First, we compute some coefficients of~$(c_n^{\chi})_n$ and~$(a_n^{\chi})_n$. 
Observe:
\begin{equation*}
gocha^\ast(\AA,t):=\frac{1}{1-d{\chi_0} t}, \quad \text{ and } \quad gocha_{\chi_0}(\AA,t):=\frac{1}{1-dt}.
\end{equation*}
Then~$c_{2n}^{\mathds{1}}=d^{2n}, \quad c_{2n}^{\chi_0}=0, \quad c_{2n+1}^{\chi_0}=d^{2n+1}, \quad \text{ and } \quad c_{2n+1}^{\mathds{1}}=0$. 
\\Moreover, 
\begin{equation*}
\log(gocha^\ast(\AA,t)):=\sum_n \frac{(d{\chi_0})^n}{n} t^n, \quad \log(gocha_{\chi_0}(\AA,t)):=\sum_{n\in \NN}\frac{d^n}{n}t^n.
\end{equation*}
So,~$b_{2n+1}^{\chi_0}:=d^{2n+1}/(2n+1), \quad b_{2n}^{\chi_0}=0,\quad b_{2n}^{\mathds{1}}=d^{2n}/(2n), \quad \text{and}\quad  b_{2n+1}^{\mathds{1}}=0$.
\\For instance, if we apply Formula \eqref{ap}, one obtains when~$p\neq 3$:
\begin{equation*}
a_3^{\chi_0}=\frac{d^3-d}{3}, \text{ and } a_3^{\mathds{1}}=0.
\end{equation*}
If we apply Proposition \ref{Weigen}, we obtain:
$$a_2^{\chi_0}=0, \quad \text{and} \quad a_2^{\mathds{1}}=\frac{d^2-d}{2}.$$

Observe that~$c_{\chi_0, n}=d^n \quad \text{and} \quad b_{\chi_0, n}:=d^n/n$. Theorem \ref{freecomp2}, allows us to check that every eigenspace of~$\Ll(\AA, G)$ is indeed infinite dimensional.
\end{exem}

\subsubsection{Non-free case}\label{parlyn}

Let us now construct some non-free examples that illustrate Theorem \ref{mildcomp2D}. For this purpose, consider~$\Delta$ a subgroup of~$Aut(F)$. We construct here a finitely presented pro-$p$ quotient~$G$ of~$F$, such that~$\Delta$ induces a subgroup of~$Aut(G)$.

We remind that~$F$ is the free pro-$p$ group generated by~$\{x^{\chi}_j\}_{\chi\in {\rm Irr}(\Delta); 1\leq j\leq d^{\chi}}$ and define~$\Ff$ the free abstract group generated by the family~$\{x^{\chi}_j\}_{\chi;j}$. Assume also that the action of~$\Delta$ is diagonal over~$\{x^{\chi}_j\}$, i.e. for all~$\delta$ in~$\Delta$,~$\delta(x^{\chi}_j)=(x^{\chi}_j)^{\chi(\delta)}$.

\begin{defi}[Comm-family]
The family~$(l_j)_{j\in [\![1;r]\!]}\subset \Ff$ is said to be a comm-family if:~$$l_j:=\prod_{l=1}^{\eta_j} u_{j,\gamma_l}^{\alpha_{j,\gamma_l}}\in F,$$
where~$\gamma_l$ and~$\alpha_{j,\gamma_l}$ are integers, and~$u_{j,\gamma_l}$ is a~$\gamma_l$-th commutator on~$\{x_j^{\chi}\}_{\chi;j}$, i.e.~$u_{j,\gamma_l}:=\lbrack x_1;\dots; x_{\gamma_l}\rbrack$ where~$x_i\in \{x_j^{\chi}\}_{\chi;j}$.
\end{defi}

\begin{prop} \label{com}
Let~$(l_j)_{j\in [\![1;r]\!]}$ be a comm-family, and denote by~$R$ its normal (topological) closure in~$F$. Then for all~$\delta$ in~$\Delta$,~$\delta(R)=R$ thus~$\Delta$ induces a subgroup of~$Aut(F/R)$.
\end{prop}

\begin{proof}
First of all, if~$u$ and~$v$ are elements in~$F$, we write~$u^v:=v^{-1}uv$. 
\\Assume~$\lbrack x;y \rbrack \in R$, where~$x$ and~$y$ are elements in~$\{x_j^{\chi}\}_{\chi;j}$. Observe the following identity:
$$1=\lbrack x;yy^{-1}\rbrack = \lbrack x;y^{-1} \rbrack \lbrack x;y \rbrack^{y^{-1}}.$$
Therefore~$\lbrack x;y^{-1}\rbrack$ is in~$R$. Remark also for all integers~$a$:
$$\lbrack x;y^a\rbrack= \lbrack x;y^{a-1}\rbrack \lbrack x;y\rbrack^{y^{a-1}}.$$ 
Thus by induction, we see that for all~$a\in \Z$, the commutator~$\lbrack x;y^a\rbrack$ is in~$R$.
\\Finally, for all integers~$b$, we also have:
$$\lbrack x^b;y \rbrack =\lbrack x;y\rbrack^{x^{b-1}} \lbrack x^{b-1};y\rbrack.$$
We conclude as before that~$\lbrack x^b;y\rbrack \in R$, for all integers~$b$.
\\Then~$\delta(R)=R$, for every~$\delta\in \Delta$.
\end{proof}

\begin{exem}\label{introex}
Here assume~$q$ is an odd prime that divides~$p-1$. Take~$F$ a free pro-$p$ group with three generators: 
$\{x^{\chi_0}_1 ,x^{\chi_0^2}_1,x^{\chi_0^3}_1\}$. Assume also that~$\Delta$ acts diagonally on the previous set. 

Consider~$R$ the closed normal subgroup of~$F$ generated by commutators~$l_1:=\lbrack x^{\chi_0}_1;x^{\chi_0^2}_1\rbrack$ and~$l_2:=\lbrack x^{\chi_0}_1;x^{\chi_0^3}_1\rbrack$. By Proposition \ref{com}, the group~$\Delta$ induces a subgroup of~$Aut(G)$. Since~$G$ is mild (see for instance \cite{FOR}), we have~$\cd(G)=2$ and:
$$gocha_{\chi_0}(\F_p,t)=\frac{1}{{\chi}_{eul, \chi_0}(\F_p, t)}=\frac{1}{1-t-t^2+t^4}.$$
Thus by Theorem \ref{mildcomp2D}, we conclude that every eigenspace of~$\Ll(\F_p, G)$ is infinite dimensional.
\end{exem}

\subsection{FAB quadratic mild examples}

Let~$\K$ be a quadratic imaginary extension over~$\Q$, with class number coprime to~$p$. Denote by~$S:=\{\p_1;\dots;\p_d\}$ a finite set of tame places of~$\K$, i.e. for~$\p\in S$,~$N_{\K/\Q}(\p) \equiv 1 \pmod p$, and assume that~$S$ is stable by~$\Delta$. We define~$\K_S$ the~$p$-maximal unramified extension of~$\K$ outside~$S$. Set~$G:=\Gal(\K_S/\K)$ and~$\Delta:=\Gal(\K/\Q)$. Again, fix~$\chi_0$ the non-trivial character of~$\Delta$ over~$\F_p$. The group~$\Delta$ acts on~$G$, and thanks to Class Field Theory, the group~$G$ has the FAB property: every open subgroup has finite abelianization. 

Write $U_\p$ for the unit group of the completion of $\K$ at the place $\p \in S$. We define the element~$X_\p \in \E_1(\F_p,G)$ as the image, given by Class Field Theory, of a generator of~$U_\p/U_\p^p$. Then (see for instance \cite[Theorem $2.6$]{rougnant2017propagation}), the set $\{X_{\p}\}_{\p\in S}$ is a basis of $\E_1(\F_p,G)$.

Denote by~$x_\p$ an element in~$G$ that lifts~$X_\p$. We introduce~$F$, the free pro-$p$ group generated by~$x_\p$. Koch \cite[Chapter~$11$]{Koch} gave a presentation of $G$, with generators~$\{x_\p\}_{p\in S}$ and relations~$\{l_{\p}\}_{\p\in S}$ verifying:
\begin{equation*} \label{relations} l_{\p_i} \equiv \prod_{j\neq i}[x_{\p_i},x_{\p_j}]^{a_j(i)} \pmod{F_3(\F_p)}, \quad \text{where} \quad a_j(i) \in \Z/p\Z.
\end{equation*}
The element $a_j(i)$ is zero if and only if the prime $\p_i$ splits in $\k^p_{\{\p_j\}}/\k$, where $k^p_{\{\p\}}$ is the (unique) cyclic extension of degree $p$ of $k$ unramified outside $\p$. This is equivalent to $$p_i^{(p_j-1)/p}\equiv 1 \pmod{p_j},$$ 
where $p_i$ is a prime in $\Q$ below $\p_i$.

From now, we assume that this presentation is \textbf{mild and quadratic} (the relations are all of weight~$2$), which means that we have the following isomorphisms of~$\F_p\lbrack \Delta\rbrack$-modules:
$$\E_1(\F_p)= \bigoplus_{i=1}^d X_{\p_i} \F_p, \quad \text{and} \quad R(\F_p)\simeq \bigoplus_{i=1}^d \left(\sum_{j\neq i} a_j(i) \lbrack X_{\p_j}; X_{\p_i} \rbrack\right) \F_p.$$
Denote by~$i$ (resp.~$s$), the number of inert or totally ramified (resp. totally split) primes below~$S$ in~$\Q$, then~$d=r=|S|:=i+2s$. Recall that for every~$\chi$:
$$d^\chi:=\dim_{\F_p} \E_1(\F_p)\lbrack \chi \rbrack,\quad \text{and} \quad r^\chi:=\dim_{\F_p} R(\F_p)\lbrack \chi \rbrack.$$
By \cite[Theorem~$1$]{Gras} and Class Field Theory, we obtain:
$$d^{\mathds{1}}=i+s \quad (\text{resp. } r^{\mathds{1}}=i+s) \quad \text{and} \quad d^{\chi_0}=s \quad (\text{resp. } r^{\chi_0}=s).$$

\begin{prop}\label{numbtheo}
We have the following equalities of series:
\begin{equation*}
\begin{aligned}
gocha^\ast(\F_p,t):=& \frac{1}{1-(i+s+s\chi_0)t+(i+s+s\chi_0)t^2}, 
\\ gocha_{\chi_0}(\F_p,t):=& \frac{1}{1-st-it^2+(s+i)t^4}.
\end{aligned}
\end{equation*}
Consequently, the action of~$\Delta$ on~$G$ is not trivial if and only if at least one place above~$S$ in~$\Q$ totally splits in~$\K$.
\end{prop}

\begin{proof}
Here, the relations have all weight~$2$, so:
$$\chi_{eul}^\ast(t):=1-(d^{\mathds{1}}+d^{\chi_0}\chi_0)t+(r^{\mathds{1}}+r^{\chi_0}\chi_0)t^2=1-(i+s+s\chi_0)t+(i+s+s\chi_0)t^2.$$
Since the presentation is mild, we conclude using Proposition \ref{key}.
\end{proof}

\begin{rema}
Before giving examples, let us add some complements.
\begin{itemize}
\item The~$\F_p\lbrack \Delta\rbrack$-module structure of~$\E_1(\F_p)$ (or~$R(\F_p)$) gives us the integers~$i$ and~$s$.
\item If every place~$\p$ above~$S$ is inert or totally ramified in~$\K$, then~$\Gal(\Q_{S}/\Q)$ and~$G:=\Gal(\K_S/\K)$ admit the same number of generators. Then Gras \cite[Theorem~$1$]{Gras}, showed that~$\Gal(\Q_S/\Q)$ and~$G$ are isomorphic, so the action of~$\Delta$ over~$G$ is trivial.
\item Assume now that all places in~$\Q$ below a set of primes~$S$ are totally split in~$\K$. If~$\Gal(\Q_S/\Q)$ is mild, Rougnant in \cite[Théorème~$0.3$]{rougnant2017propagation} gave a criterion to also obtain~$\Gal(\K_S/\K)$ mild.
\end{itemize}

\end{rema}

\begin{exem}\label{aritex}
We give explicit arithmetic examples where~$G$ is mild and defined by quadratic relations:

\begin{enumerate}
\item 
We study the following example given by \cite[Example~$3.2$]{epub559}: let~$p=3$,~$\K:=\Q(i)$, and consider the set of primes: $S:=\{q_1:=229, q_2:=241\}$. These primes totally split in~$\K$. and the places above~$S$ (in $\K$) are given by:
$$S:=\{\p_1:=(2+15i), \overline{\p_1}:=(2-15i), \p_2:=(4+15i), \overline{\p_2}:=(4-15i)\}.$$
The group~$G:=\Gal(\K_S/\K)$ is mild quadratic. Then by Proposition \ref{numbtheo}:
$$gocha^\ast(\F_p,t)=\frac{1}{1-(2+2\chi_0)t+(2+2\chi_0)t^2}, \quad \text{and} \quad gocha_{\chi_0}(\F_p,t)=\frac{1}{1-2t+2t^4}.$$
However, the polynomial~$1-2t+2t^4$ admits only non real roots, so we can not apply Theorem \ref{mildcomp2D}.

Observe by \cite[Example~$11.15$]{Koch}, that the group~$\Gal(\Q_{S}/\Q)$ is finite. 

\item 
\cite[Part~$6$]{rougnant2017propagation}: Take~$p=3$,~$\K:=\Q(\sqrt{-5})$, and~$S:=\{61; 223; 229; 481\}$. The Class group of~$\K$ is~$\Z/2\Z$, the primes in~$S$ are totally split in~$\K$, and the groups~$\Gal(\Q_{S}/\Q)$ and~$G:=\Gal(\K_S/\K)$ are both mild quadratic. Therefore, by Proposition \ref{numbtheo}, we obtain:
$$gocha^\ast(\F_p,t)=\frac{1}{1-(4+4\chi_0)t+4\chi_0 t^2} \quad \text{and} \quad gocha_{\chi_0}(\F_p,t)=\frac{1}{1-4t+4t^4}.$$
By Theorem \ref{mildcomp2D}, the graded spaces~$\Ll(\F_p, G)\lbrack \chi_0 \rbrack$ and~$\Ll(\F_p, G)\lbrack \mathds{1} \rbrack$ are both infinite dimensional.

\item 
We enrich the example given in \cite[Part~$2.1$]{split}: Consider~$p=3$,~$\K:=\Q(\sqrt{-163})$, and~$T:=\{31,19,13,337,7\}$. The class group of~$\K$ is trivial,~$\Gal(\Q_T/\Q)$ is mild, and the primes in~$T$ are inert in~$\K$. Therefore by \cite[Theorem~$1$]{Gras}, the group~$\Gal(\K_T/\K)$ is mild (in fact, it has the same linking coefficients as~$\Gal(\Q_T/\Q)$).

Observe that~$43$ is totally split in~$\K$, so we take~$\{\p_6, \overline{\p_6}\}$ to be the primes in~$\K$ above~$43$. 
Consider now~$S:=T\cup\{\p_6;\overline{\p_6}\}$. 
By \cite[Corollary $4.3$]{epub559}, the group~$G:= \Gal(\K_S/\K)$ is mild quadratic.
Proposition \ref{numbtheo} gives us
$$gocha^\ast(\F_p,t):=\frac{1}{1-(6+\chi_0) t+ (6+\chi_0)t^2}, \quad \text{and} \quad gocha_{\chi_0}(\F_p,t):=\frac{1}{1-t-5t^2+6t^4}.$$
Therefore, by Theorem \ref{mildcomp2D}, the graded spaces~$\Ll(\F_p, G)\lbrack \mathds{1}\rbrack$ and~$\Ll(\F_p, G)\lbrack \chi_0 \rbrack$ are infinite dimensional.
\end{enumerate}

\end{exem}

\section*{Remark on lower~$p$-central series and mild groups}
Assume here that~$G$ is a finitely presented pro-$p$ group, and~$q$ divides~$p-1$. We define the lower~$p$-central series of~$G$ by:
$$G_{\{1\}}:=G,\quad \text{and} \quad G_{\{n+1\}}:=G_{\{n\}}^p\lbrack G_{\{n\}};G\rbrack.$$
Remark that~$\bigoplus_{n\in \NN}(G_{\{n\}}/G_{\{n+1\}})$ is an~$\F_p\lbrack t \rbrack \lbrack \Delta \rbrack$-module, where~$\F_p\lbrack t \rbrack$ is the ring of polynomials over~$\F_p$.

Furthermore, if we assume~$G$ mild (see \cite[Definition~$1.1$]{Labute}), Labute showed in \cite[Part~$4$]{Labute}, that the lower~$p$-central series come from the filtered algebra defined by~$Al(\Z_p,G)$ endowed with the filtration induced by~$\{Al_{\{n\}}(G):=\ker(Al(\Z_p,G)\to \F_p)^n\}_{n\in \NN}$. Additionally, the set~$\bigoplus_{n\in \NN}(G_{\{n\}}/G_{\{n+1\}})$ is a free~$\F_p\lbrack t\rbrack$-module. Since~$G$ is finitely generated, we introduce:
$$a_{\{n\}}^{\chi}:=\rk_{\F_p}(G_{\{n\}}/G_{\{n+1\}})\lbrack \chi \rbrack,\quad \text{and} \quad c_{\{n\}}^{\chi}:=\rk_{\F_p}(Al_{\{n\}}(G)/Al_{\{n+1\}}(G))\lbrack \chi \rbrack.$$

If we replace~$a_n(\Z_p)$ (resp.~$c_n(\Z_p)$) by~$a_{\{n\}}$ (resp.~$c_{\{n\}}$), then the results of this paper can be adapted for lower~$p$-central series. Moreover, extending \cite[Corollary~$2.7$]{Labute} in an equivariant context, we can deduce a relation between the coefficients~$c_n^{\chi}$ and~$a_{\{n\}}^{\chi}$.

\bibliography{bibactfi}
\bibliographystyle{plain}
\end{document}